\documentclass{amsart}
\usepackage{amsmath}
\usepackage{latexsym}
\usepackage{color}
\usepackage{amscd}
\usepackage[all]{xy}
\usepackage{enumerate}
\usepackage{soul}
\usepackage{graphicx}
\usepackage{comment}
\newtheorem{thm}{Theorem}[section]
\newtheorem{lem}[thm]{Lemma}
\newtheorem{cor}[thm]{Corollary}
\newtheorem{prop}[thm]{Proposition}
\newtheorem{defn}[thm]{Definition}
\newtheorem{problem}[thm]{Problem}
\newtheorem{rem}[thm]{Remark}

\newtheorem{mthm}{Theorem}

\theoremstyle{definition}

\def\CM{\mathcal M}

\def\CH{\mathcal H}

\def\cl{\mathrm{cl}}
\def\scl{\mathrm{scl}}
\def\e{\equiv_P}

\begin{document}

\title[Signatures of surface bundles and stable commutator lengths of Dehn twists]
{Signatures of surface bundles and stable commutator lengths of Dehn twists}

\author[N. Monden]{Naoyuki Monden}
\address{Department of Mathematics, Faculty of Science, Okayama University, Okayama 700-8530, Japan}
\email{n-monden@okayama-u.ac.jp}

\begin{abstract}
The first aim of this paper is to give four types of examples of surface bundles over surfaces with non-zero signature. 
The first example is with base genus 2, a prescribed signature, a 0-section and the fiber genus greater than a certain number which depends on the signature. 
This provides a new upper bound on the minimal base genus for fixed signature and fiber genus. 
The second one gives a new asymptotic upper bound for this number in the case that fiber genus is odd. 
The third one has a small Euler characteristic. 
The last is a non-holomorphic example.

The second aim is to improve upper bounds for stable commutator lengths of Dehn twists by giving factorizations of powers of Dehn twists as products of commutators. 
One of the factorizations is used to construct the second examples of surface bundles. 
As a corollary, we see that there is a gap between the stable commutator length of the Dehn twist along a nonseparating curve in the mapping class group and that in the hyperelliptic mapping class group if the genus of the surface is greater than or equal to $8$. 
\end{abstract}

\maketitle

\setcounter{secnumdepth}{2}
\setcounter{section}{0}

\section{Introduction}\label{section:one}
\subsection{Notation}
In here, we introduce notation. 
Let $\Sigma_g^r$ be a compact oriented surface of genus $g$ with $r$ boundary components, and let $\CM_g^r$ be the mapping class group of $\Sigma_g^r$, that is the group of isotopy classes of orientation preserving self-diffeomorphisms of $\Sigma_g^r$ such that diffeomorphisms and isotopies fix the points of the boundaries. 
For simplicity, we write $\Sigma_g=\Sigma_g^0$ and $\CM_g=\CM_g^0$. 
For a subsurface $\Sigma$ of $\Sigma_g^r$, let $\CM(\Sigma)$ denote the subgroup of $\CM_g^r$ generated by elements whose the restrictions on $\Sigma_g^r-\Sigma$ are identity. 
We denote by $i(a,b)$ the geometric intersection number for two simple closed curves $a$ and $b$ on $\Sigma_g^r$.

For two elements $\phi_1,\phi_2$ in $\CM_g^r$, the notation $\phi_2\phi_1$ means that we first apply $\phi_1$ then $\phi_2$, the conjugation $\phi_2 \phi_1 \phi_2^{-1}$ of $\phi_1$ by $\phi_2$ is denoted by ${}_{\phi_2}(\phi_1)$, and we write $[\phi_1,\phi_2]$ for the commutator of $\phi_1$ and $\phi_2$. 
We denote by $t_c$ the right-handed Dehn twist along a simple closed curve $c$ on $\Sigma_g^r$. 
Since $\CM_g^r$ is generated by Dehn twists \cite{de}, every element can be written as a product of Dehn twists. 
If we consider an element in $\CM_g^r$ without explicit factorization as a product of Dehn twist, then we suppose that its certain factorization is given and fixed.

A surface bundle over a surface is a fiber bundle that the fiber and the base are closed oriented surfaces. 
If the fiber and the base are $\Sigma_g$ and $\Sigma_h$, respectively, then we call this the $\Sigma_g$-bundle over $\Sigma_h$. 
For the total space $X$ of this bundle, we denote by $\sigma(X)$ the signature of $X$. 
We write it simply $\sigma$ when no confusion can arise.

In this paper, we introduce the symbol ``$\e$" in Section~\ref{signature}. 
If the reader is interested only in the results on the (stable) commutator length, then he or she may replace ``$\e$" by ``$=$" and skip Section~\ref{global}, \ref{signature}, \ref{Surface bundles with base of genus two} and~\ref{odd}.

\subsection{Surface bundles over surfaces with non-zero signature}
Even though to consider surface bundles over surfaces is one simple way to get 4-manifolds, many fundamental problems on such bundles remain open. 
Problems about surface bundles with non-zero signature are exemplified as one of them.

Euler characteristics multiply for fiber bundles. 
In contrast, this property does not hold for the signature. 
Equivalently, there is a surface bundle over a surface with non-zero signature. 
Such examples were first exhibited by Atiyah \cite{Atiyah} and, independently, Kodaira \cite{Kodaira}. 
Since then, many examples of surface bundles with nonvanishing signature have been constructed (see e.g. \cite{Hirzebruch,Endo,BDS,BD,Stipsicz,EKKOS,Akhmedov,Lee}).

A $\Sigma_g$-bundle over $\Sigma_h$ gives some restrictions on the signature $\sigma$. 
For example, $\sigma$ must be divisible by $4$, and it vanishes if $h\leq 1$ or $g\leq 2$ using Meyer's signature cocycle and Birman-Hilden's relations \cite{bh} of $\CM_2$ (see \cite{Meyer1,Meyer2}). 
Hence, the case of $g\geq 3$ and $h\geq 2$ is interesting. 
The existence of an example of $g=3$ and $\sigma\neq 0$ was shown in \cite{Meyer1,Meyer2}, and explicit examples were constructed in \cite{Endo,Stipsicz,EKKOS,Lee}. 
In particular, for any integer $n$, there is a $\Sigma_3$-bundle over $\Sigma_h$ with $\sigma=4n$ if $h\geq 7|n|+1$ (see \cite{Lee}). 
An example of $h=2$ and $\sigma\neq 0$, which solves Problem 2.18 (A) in \cite{Kirby}, was first given by Bryan-Donagi \cite{BD}. 
Precisely, it satisfies $g=4k^3-2k^2+1$ and $\sigma=8(k^3-k)/3$ for any integer $k\geq 2$.
Thus, we notice that $g$ and $\sigma$ in the example of $h=2$ take discrete values compared to $h$ and $\sigma$ in the examples of $g=3$. 
If the example of \cite{BD} has a $0$-section (i.e. a section of self-intersection zero), then  the genus of a fiber can extend to $g\geq 4k^3-2k^2+1$ using ``section sum operations". 
However, the author does not know whether it admits a $0$-section or not. 
The motivation for the next result comes from these observations. 
\begin{mthm}\label{thmA}
For any integer $n$, there is a $\Sigma_g$-bundle over $\Sigma_2$ with $\sigma=4n$ if $g\geq 39|n|$. 
In particular, it admits a $0$-section. 
\end{mthm}

Meyer \cite{Meyer1,Meyer2} also proved that for every $g\geq 3$ and $n$, there is a $\Sigma_g$-bundle over $\Sigma_h$ with $\sigma=4n$. 
Motivated by this result, Problem~\ref{problem1} below, which is a refined version of Problem 2.18 (A) in \cite{Kirby}, was posed by Endo \cite{Endo}. 
Solving Problem~\ref{problem1} is equivalent to computing the minimal genus of the surfaces representing the $n$ times generator of $H_2(\CM_g;\mathbb{Z})/\mathrm{Tor}$ for fixed $g\geq 3$ and $n$ (see \cite{k1}). 
\begin{problem}[Endo \cite{Endo}]\label{problem1}
Let $h_g(n)$ be the minimal $h$ such that there exists a $\Sigma_g$-bundle over $\Sigma_h$ with $\sigma=4n$. 
Determine the value $h_g(n)$. 
\end{problem}
Upper bounds on $h_g(n)$ were given in \cite{Endo} after the initial work in \cite{Stipsicz,EKKOS,Lee}. 
A sharper bound given by Lee \cite{Lee} is $h_g(n) \leq 5|n|+1$ for $g\geq 6$. 
As a corollary of Theorem~\ref{thmA}, we can compute $h_g(n)$ for the special case and give it's upper bound for $g\geq 39$ by pulling back the bundle to unramified coverings of $\Sigma_2$ of degree $|n|$. 
\begin{cor} 
For any $n$, $h_g(n)=2$ if $g\geq 39|n|$, and $h_g(n)\leq |n|+1$ if $g\geq 39$. 
\end{cor}

Kotschick \cite{k1} first gave the lower bound on $h_g(n)$. 
The best known bound was obtained by Hamenstadt \cite{Hamenstadt}: $3|n|/(g-1)+1 \leq h_g(n)$. 
Since the upper bound with the same shape as the above lower bound, in which $g$ appears in the denominator, was given in \cite{EKKOS}, we next turn to study the asymptotic behavior of $h_g(n)$. 
This is natural since the base genus and the signature grow linearly in a sequence of bundles by pulling back by covers of the base of a given bundle. 
We consider the following problem posed by Mess (see Problem 2.18 (B) in \cite{Kirby}). 
\begin{problem}[Mess \cite{Kirby}]\label{problem2}
Let $\displaystyle H_g: = \lim_{n\to \infty} \frac{h_g(n)}{n}$. Determine the limit $H_g$. 
\end{problem}
The limit exists and is finite and interpreted as the Gromov-Thurston norm of the generator of $H_2(\CM_g;\mathbb{Z})/\mathrm{Tor}$ (see \cite{k1}). 
The lower bound $3/(g-1)\leq H_g$ is immediately obtained from the result of \cite{Hamenstadt}. 
For any $g\geq 3$, the upper bound on $H_g$ was first given in \cite{EKKOS}. 
This bound was improved as follows: $H_g\leq 6/(g-2)$ for even $g$, $H_g\leq 9/(g-2)$ for $g=3k\geq 6$ and $H_g\leq 14/(g-1)$ for odd $g$ (see \cite{BD,BDS,Lee}). 
Since there is a gap between the even and odd $g$ cases, we fill it. 
\begin{mthm}\label{thmB}
If $g$ is odd, then, for any integer $n$, there is an $\Sigma_g$-bundle over $\Sigma_{6|n|+5}$ with $\sigma=4(g-1)n$. 
Therefore, $H_g \leq 6/(g-1)$ for odd $g$. 
\end{mthm}

We next focus on surface bundles over surfaces with small Euler characteristics. 
The Euler characteristic of a $\Sigma_g$-bundle over $\Sigma_h$ is $4(g-1)(h-1)$. 
The smallest known example is that of \cite{Lee} ($g=3$, $h=8$ and $\sigma=4$). 
We slightly improve it. 
\begin{mthm}\label{thmC}
There exists a $\Sigma_3$-bundle over $\Sigma_7$ with $\sigma=4$ and a $0$-section. 
\end{mthm}

Finally, we give non-holomorphic examples with non-zero signature. 
Thurston \cite{Thurston} showed that the total space of a $\Sigma_g$-bundle over $\Sigma_h$ is symplectic for $g\geq 2$. 
Then, the following question arises: \textit{for which values of $g$ and $h$ the total space of a $\Sigma_g$-bundle over $\Sigma_h$ with $\sigma\neq 0$ does not admit a complex structure?} 
If a holomorphic surface bundle is isotrivial, then $\sigma=0$ (see \cite{BD}), and there are simple examples with $\sigma=0$ that is non-isotrivial and whose total space can not be complex (see \cite{Baykur}). 
From this, we need the assumption on $\sigma\neq 0$. 
Baykur \cite{Baykur} showed that for any positive integer $N$ and for any $h\geq 3$, there exists $g>N$ such that there are infinite families of (pairwise non-homotopic) 4-manifolds with $\sigma \neq 0$ admitting a $\Sigma_g$-bundle over $\Sigma_h$ and not admitting any complex structure with either orientation (The same holds for any $g\geq 4$ if $h\geq 9$). 
Using Theorem 4 (2) of \cite{Baykur} and Theorem~\ref{thmA}, we see that the same is true for $h=2$ (i.e. the smallest $h$ satisfying $\sigma\neq0$). 
\begin{cor}
For any integer $n$ and for any $g\geq 39|n|+1$, there are infinite families of (pairwise non-homotopic) 4-manifolds with $\sigma=4n$ admitting a $\Sigma_g$-bundle over $\Sigma_2$ and not admitting any complex structure with either orientation. 
\end{cor}

\subsection{Stable commutator lengths of Dehn twists}\label{scl}
Since the monodromy factorization of a $\Sigma_g$-bundle over $\Sigma_h$ is a factorization of the identity as a product of $h$ commutators in $\CM_g$, techniques constructing commutators and reducing the number of them are required to prove Theorem~\ref{thmA}, \ref{thmB} and \ref{thmC}. 
We apply the techniques of (stable) commutator lengths on $\CM_g$ to the results on surface bundles. 
Especially, Theorem~\ref{thmD} (1) below will be used to show Theorem~\ref{thmB}.

Let $[G,G]$ be the commutator subgroup of a group $G$. 
For $x \in [G,G]$, the \textit{commutator length} $\cl_{G} (x)$ of $x$ is defined to be the smallest number of commutators whose product is equal to $x$. 
The \textit{stable commutator length} $\scl_{G} (x)$ of $x$ is the limit 
\[\scl_{G} (x) = \lim_{ n \to \infty} \frac{ \cl_{G} (x^n ) }{n}. \]
Note that the limit exists. 
We define $\cl_G(x) := \infty$ if $x\notin [G,G]$, $\scl_G(x) := \scl_G(x^k)/|k|$ if $x\notin[G,G]$ but $x^k \in [G,G]$ for some $k$ and $\scl_G(x):=\infty$ if $x^k\notin [G,G]$ for any $k$. 
From the results of \cite{bh} and \cite{po}, $\scl_{\CM_g}(x) <\infty$ for $g\geq 1$. 
Since Dehn twists are the most fundamental generators of $\CM_g$, computing $\cl_{\CM_g}(t_c)$ and $\scl_{\CM_g}(t_c)$ is the natural problem. 
Korkmaz and Ozbagci \cite{ko1} showed that $\cl_{\CM_g}(t_c)=2$ if $g\geq 3$. 
Therefore, our next problem is to the calculate $\cl_{\CM_g}(t_c^n)$ for any $n$ and $\scl_{\CM_g}(t_c)$. 
However, since it is difficult to compute $\cl_G$ and $\scl_G$ in general, it makes sense to give estimates on $\cl_{\CM_g}(t_c^n)$ and $\scl_{\CM_g}(t_c)$.

A lower bound on $\scl_{\CM_g}(t_c)$ was given by Endo-Kotschick \cite{ek}. 
Consequently, $\CM_g$ is not uniformly perfect, and the natural homomorphism from the second bounded cohomology of $\CM_g$ to its ordinary cohomology is not injective, which were conjectured by Morita \cite{m}. 
For technical reasons, they showed that $|n|/(18g-6) + 1 \leq \cl_{\CM_g}(t_c^n)$ for any $n$ if $c$ is a separating curve. 
This gives $1/(18g-6) \leq \scl_{\CM_g}(t_c)$ for a separating curve $c$. 
This assumption was removed by Korkmaz \cite{ko}, and the above results were extended to positive multi twists in \cite{BK}. 
In \cite{ko}, an upper bound on $\scl_{\CM_g}(t_c)$ was also given. 
He showed that $\scl_{\CM_g}(t_c)<2/30$ for a nonseparating curve $c$ if $g\geq 2$. 
On the other hand, there is an estimate $\scl_{\CM_g}(t_c) = O(1/g)$ for any simple closed curve $c$, so $\lim_{g \to \infty} \scl_{\CM_g}(t_c) = 0$ (see \cite{k2} and also \cite{ca}). 
Such an explicit upper bound on $\scl_{\CM_g}(t_c)$ was given in \cite{cms} if $c$ is nonseparating, and in \cite{my} if $c$ is separating. 
However, they don't give an factorization of $t_c^n$ as a product of commutators realizing $\lim_{g\to \infty} \scl_{\CM_g}(t_c)=0$ explicitly.

The purpose is to give sharper upper bounds for stable commutator lengths of Dehn twists giving explicit factorizations of powers of Dehn twists as products of commutators. 
We call a simple closed curve $s$ on $\Sigma_g$ the \textit{separating curve of type $h$} if $s$ separates into two components with genera $h$ and $g-h$ for $h=1,2,\ldots,[\frac{g}{2}]$. 
To state our results, let $s_0$ be a nonseparating curve on $\Sigma_g$ and let $s_h$ a separating curve of type $h$ on $\Sigma_g$. 
Our main results are following. 
\begin{mthm}\label{thmD}
Let $g\geq 2$ and $h\geq 2$. For any integer $n$, we have the following. 
\begin{enumerate}
\item[{\rm (1)}] $\cl_{\CM_g}(t_{s_0}^{10(g-1)n})\leq |n|+3$, and therefore $\scl_{\CM_g}(t_{s_0}) \leq 1/(10g-10)$, 
\item[{\rm (2)}] $\cl_{\CM_g}(t_{s_1}^{5(g-1)n})\leq [(7|n|+9)/2]$, and therefore $\scl_{\CM_g}(t_{s_1}) \leq 7/(10g-10)$, 
\item[{\rm (3)}] $\cl_{\CM_g}(t_{s_h}^{[g/h]n})\leq [(|n|+3)/2]$, and therefore $\scl_{\CM_g}(t_{s_h}) \leq 1/(2[g/h])$. 
\end{enumerate}
In particular, there are factorizations of powers of Dehn twists as products of commutators realizing the above upper bounds for the commutator lengths. 
\end{mthm}
Sharper upper and lower bounds were given in \cite{mo,cms,my} if $g=2$.

Let $\CH_g$ be the hyperelliptic mapping class group of $\Sigma_g$, that is the subgroup of $\CM_g$ consisting of all elements that commute with isotopy class of some fixed hyperelliptic involution. 
Since $\CM_g=\CH_g$ if $g=1,2$, we have $\scl_{\CM_g}\equiv \scl_{\CH_g}$. 
In general, for a subgroup $H$ of a group $G$, we have $\scl_G(x)\leq \scl_H(x)$. 
By $1/(8g+4)\leq \scl_{\CH_g}(t_{s_0})$ (see \cite{mo}) and Theorem~\ref{thmD} (1), we obtain the following corollary. 
\begin{cor}
If $g\geq 8$, then $\scl_{\CM_g}(t_{s_0}) < \scl_{\CH_g}(t_{s_0})$.
\end{cor}

From \cite{po}, we have $\cl_{\CM_g}(t_c^n) <\infty $ for any $n$ if $g\geq 3$. 
In contrast, $\cl_{\CM_1}(t_{s_0}^n) < \infty$ (resp. $\cl_{\CM_2}(t_{s_0}^n) < \infty$ and $\cl_{\CM_2}(t_{s_1}^n) < \infty$) if and only if $n \equiv 0 \pmod{12}$ (resp. $n \equiv 0 \pmod{10}$ and $n \equiv 0 \pmod{5}$). 
Even though $\scl_{\CM_1}(t_{s_0})= 1/12$ (see Remark 4.5 in \cite{cms}), to my knowledge, $\cl_{\CM_1}(t_{s_0}^{12})$ is still unknown. 
We determine $\cl_{\CM_1}(t_{s_0}^{12n})$. 
It was shown in \cite{ko1} (resp. \cite{ks}) that $t_{s_0}^{10}$ (resp. $t_{s_1}^5$) in $\CM_2$ is written as products of $2$ commutators (resp. $6$ commutators). 
Hence, $\cl_{\CM_2}(t_{s_0}^{10}) \leq 2$ and $\cl_{\CM_2}(t_{s_1}^5)\leq 6$. 
We generalize the results to $10n$ and $5n$ and improve the result of \cite{ks} slightly. 
\begin{mthm}\label{thmE} 
For any integer $n$, we have the following. 
\begin{enumerate}
\item[{\rm (1)}] $\cl_{\CM_1}(t_{s_0}^{12n}) = |n|+1$, 
\item[{\rm (2)}] $\cl_{\CM_2}(t_{s_0}^{10n}) \leq |n|+1$, 
\item[{\rm (3)}] $\cl_{\CM_2}(t_{s_1}^{5n}) \leq [(7|n|+3)/2]$. 
\end{enumerate}
In particular, there are factorizations of powers of Dehn twists as products of commutators realizing the above upper bounds. 
\end{mthm}

\subsection{Outline}
The outline of the paper is as follows. 
In Section~\ref{section:2}, we introduce some relators in $\CM_g^r$ and a signature formula for achiral Lefschetz fibrations given by Endo-Hasegawa-Kamada-Tanaka \cite{ehkt}. 
They will be used to compute the signatures of surface bundles over surfaces. 
Section~\ref{lemmas} exhibits techniques to write certain elements as products of commutators and to reduce the number of commutators. 
In Section~\ref{section:two}--\ref{Scl of the Dehn twist along a separating curve}, we give the proofs of the main results.

\vspace{0.1in}
\noindent \textit{Acknowledgements.} 
I wishes to express my gratitude to H. Endo, S. Kamada and K. Tanaka for their explanations on \cite{ehkt} and helpful comments, to A. Akhmedov and R. I. Baykur for their interests and asking me the existence of the bundle in Theorem~\ref{thmA} and to M. Korkmaz for his comments. 
I am especially grateful to H. Endo with whom I discussed the subject matter of this paper. 
The author was supported by Grant-in-Aid for Young Scientists (B) (No. 16K17601), Japan Society for the Promotion of Science.

\section{Relators in mapping class groups and a signature formula}\label{section:2}
In this section, we present the signature formula for achiral Lefschetz fibrations given in \cite{ehkt}. 
When we consider an achiral Lefschetz fibration, we obtain its global monodromy in the mapping class group of the fiber. 
The result in \cite{ehkt} says that we can compute the signature of the total space of the fibration by ``counting the numbers of certain relators" included in the global monodromy.

The outline of this section is as follows. 
We give a brief summary of the global monodromy of an achiral Lefschetz fibration in Subsection~\ref{global}. 
In Subsection~\ref{presentation}, we describe four fundamental relators and the infinite presentation of $\CM_g$ given by Luo \cite{Luo}. 
In Subsection~\ref{signature}, we review the result of \cite{ehkt}.

\subsection{The global monodromy of an achiral Lefschetz fibration}\label{global}
We briefly describe the global monodromy and the section of an achiral Lefschetz fibration. 

Let $g\geq 2$. 
Roughly speaking, a genus-$g$ \textit{achiral Lefschetz fibration} $\pi:X\to \Sigma_h$ is a smooth fibration of a 4-manifold $X$ over $\Sigma_h$ with regular fiber $\Sigma_g$ and finitely many singular fibers. 
The singular fibers are classified two types: \textit{of type $+1$}, and \textit{of type $-1$}. 
Each singular fiber is obtained by collapsing a simple closed curve $v$ on $\Sigma_g$, called the vanishing cycle. 
Note that if $\pi$ has no singular fibers, then it is an $\Sigma_g$-bundle over $\Sigma_h$. 
When we give a genus-$g$ achiral Lefschetz fibration $X\to \Sigma_h$ with $n$ singular fibers of type $\epsilon_1,\epsilon_2,\ldots,\epsilon_n$ whose vanishing cycles are $v_1,v_2,\ldots,v_n$, where $\epsilon_i=\pm1$, we obtain the following relator (up to cyclic permutations), called the \textit{global monodromy} of $\pi$, in $\CM_g$: 
\begin{align}
t_{v_1}^{\epsilon_1} t_{v_2}^{\epsilon_2} \cdots t_{v_n}^{\epsilon_n} [\mathcal{X}_1,\mathcal{Y}_1][\mathcal{X}_2,\mathcal{Y}_2] \cdots [\mathcal{X}_h,\mathcal{Y}_h] = \mathrm{id}. \label{monodromyLF}
\end{align}
Conversely, if we give the above relator, then we get a genus-$g$ achiral Lefschetz fibration $X\to \Sigma_h$ with $n$ singular fibers of type $\epsilon_1,\epsilon_2,\ldots,\epsilon_n$ whose vanishing cycles are $v_1,v_2,\ldots,v_n$.

A genus-$g$ achiral Lefschetz fibration $\pi:X \to \Sigma_h$ with the global monodromy (\ref{monodromyLF}) admits a $(-k)$-section (that is, $s:\Sigma_h \to X$ such that $\pi \circ s = \mathrm{id}_{\Sigma_h}$ and $[s(\Sigma_h)]^2=-k$) if and only if there exists a lift of (\ref{monodromyLF}) from $\CM_g$ to $\CM_g^1$ as 
\[t_{\widetilde{v}_1}^{\epsilon_1} t_{\widetilde{v}_2}^{\epsilon_2} \cdots t_{\widetilde{v}_n}^{\epsilon_n} [\widetilde{\mathcal{X}}_1,\widetilde{\mathcal{Y}}_1][\widetilde{\mathcal{X}}_2,\widetilde{\mathcal{Y}}_2] \cdots [\widetilde{\mathcal{X}}_h,\widetilde{\mathcal{Y}}_h] = t_\partial^{k},\]
where $\partial$ is the boundary curve on $\Sigma_g^1$, $t_{\widetilde{v}_i}$ is a Dehn twist mapped to $t_{v_i}$ under the map $\CM_g^1 \to \CM_g$ induced by the inclusion $\Sigma_g^1 \to \Sigma_g$, and similarly, $\widetilde{\mathcal{X}}_j$ and $\widetilde{\mathcal{Y}}_j$ are mapped to $\mathcal{X}_j$ and $\mathcal{Y}_j$, respectively.

By the result of \cite{ehkt}, the signature of $X$ is determined by ``the numbers of certain relators" of $\CM_g$ included in (\ref{monodromyLF}). 
In the next subsection, we introduce the relators.

\subsection{Luo's infinite presentations of mapping class groups}\label{presentation}
In \cite{ehkt}, they employ an infinite presentation of $\CM_g^r$ given by Luo \cite{Luo}. 
To state it, we introduce four fundamental relators in $\CM_g^r$. 
\begin{defn}\label{relators}\rm 
Let $a,b$ be simple closed curve on $\Sigma_g^r$. 
\begin{itemize}
\item If $a$ is a trivial, then $t_a=\mathrm{id}$, so we call it the \textit{trivial relator} and write \[T:=t_a.\] 

\item Let $c=t_b(a)$. Then, we have the relation $t_c= t_bt_at_b^{-1}$, called the \textit{primitive braid relation}. 
Therefore, we obtain the \textit{primitive braid relator}
\[P:=t_c^{-1}t_bt_at_b^{-1}.\]

\item\label{2chain} Let $a,b$ be simple closed curves on $\Sigma_1^1$ bounded by $c$ with $i(a,b)=1$ as in Figure~\ref{2chaincurves}. 
Then, the \textit{2-chain relation} $t_c=(t_at_b)^6$ holds in $\CM_1^1$.
This gives the \textit{2-chain relator} 
\[C_2 := t_c^{-1} (t_at_b)^6 ,\]

\item\label{lanternrelator} Let $x,y,z$ be the interior curves on $\Sigma_0^4$ as in Figure~\ref{lanterncurves}, and let $a,b,c,d$ be the boundary curves on $\Sigma_0^4$ as in the figure. 
Then, the \textit{lantern relation} $t_a t_b t_c t_d = t_x t_y t_z$ holds in $\CM_0^4$. 
Then, we have the \textit{lantern relator}
\[L := t_d^{-1} t_c^{-1} t_b^{-1} t_a^{-1} t_x t_y t_z .\]
\end{itemize}
\end{defn}
\begin{figure}[htbp]
 \begin{tabular}{c}
  \begin{minipage}{0.5\hsize}
  \centering
       \includegraphics[scale=.20]{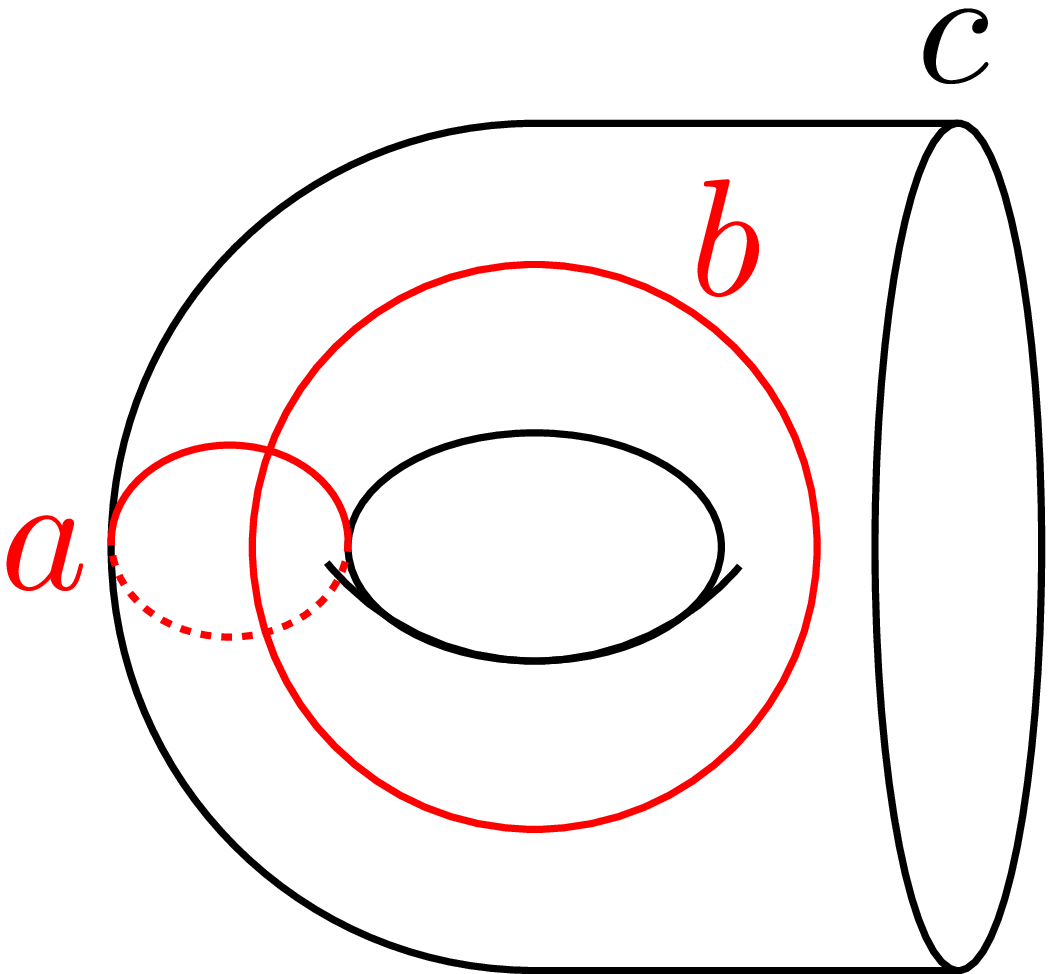}
       \caption{The curves $a,b,c$ on $\Sigma_1^1$.}
       \label{2chaincurves}
 \end{minipage}
 \begin{minipage}{0.5\hsize}
  \centering
       \includegraphics[scale=.20]{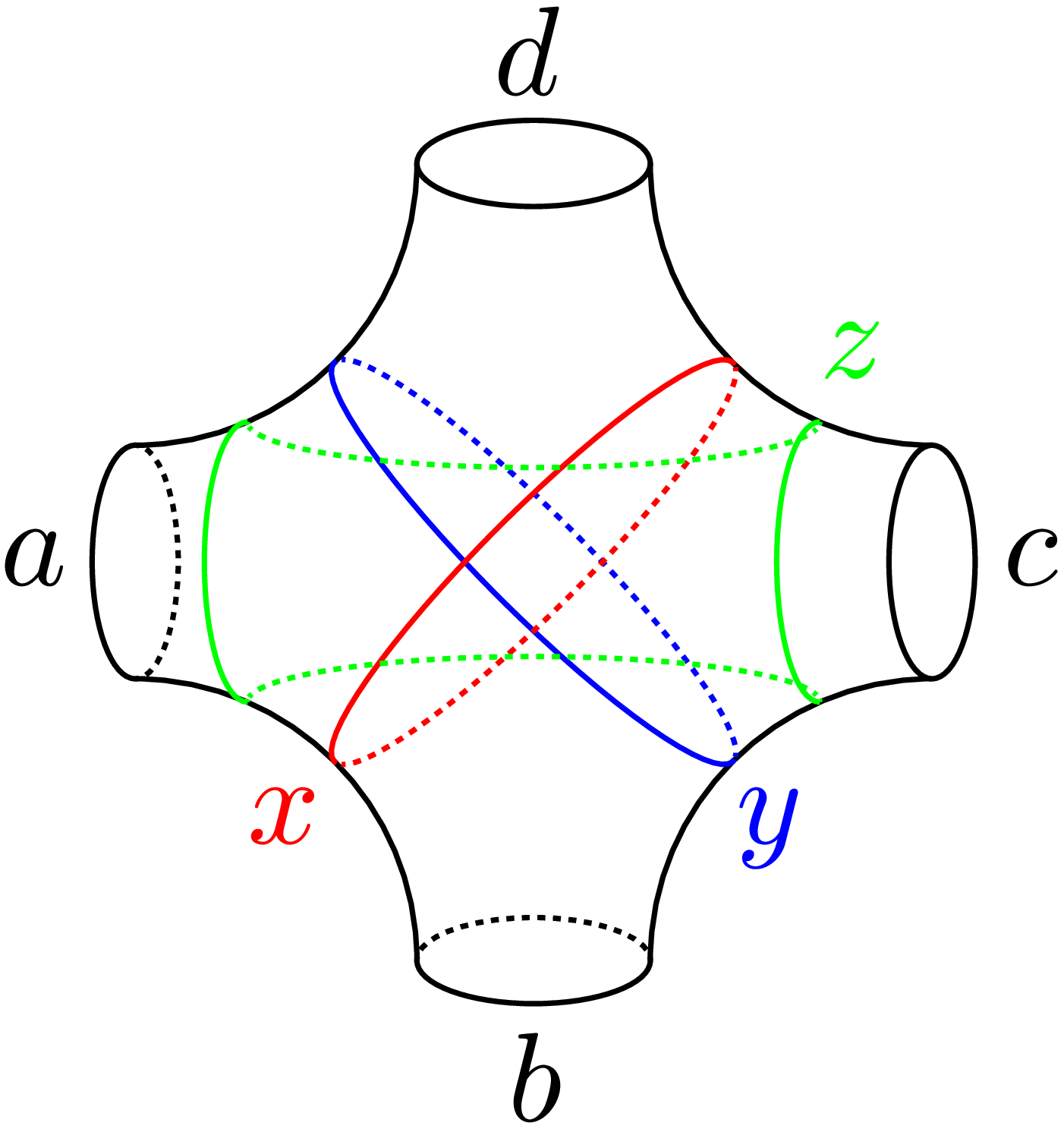}
       \caption{The curves $a,b,c,d,x,y,z$ on $\Sigma_0^4$.}
       \label{lanterncurves}
  \end{minipage}
 \end{tabular}
\end{figure}

Luo \cite{Luo} gave the following infinite presentation of the mapping class group $\CM_g^r$. 
\begin{thm}[\cite{Luo}]\label{Luo}
$\CM_g^r$ has an infinite presentation whose generators are the set of all Dehn twists and whose relators are $T$, $P$, $C_2$ and $L$. 
\end{thm}

In the rest of this subsection, we present variations of the primitive braid relator $P$. 
They are used throughout this paper. 
Before it, we give the following lemma. 
\begin{lem}\label{lemma:???}
Let $f$ be a product of $k$ Dehn twists in $\CM_g^r$. 
For $a$ a simple closed curve on $\Sigma_g^r$, $t_{f(a)}^{-1}ft_af^{-1}$ is a product of $k$ primitive braid relators. 
\end{lem}
\begin{proof}
Let $f = t_{b_k}^{\epsilon_k} \cdots t_{b_2}^{\epsilon_2} t_{b_1}^{\epsilon_1}$, where $\epsilon_i=\pm 1$ and each $b_i$ is a simple closed curve on $\Sigma_g^r$. 
For simplicity, we set 
$c_0=a$, $c_i=t_{b_i}^{\epsilon_i}(c_{i-1})$ for $i=1,2,\ldots,k$, so $c_k=f(a)$. 
Then, $P_i^{\epsilon_i}=t_{c_i}^{-1}t_{b_i}^{\epsilon_i}t_{c_{i-1}}t_{b_i}^{-\epsilon_i}$ is a primitive braid relator if $\epsilon_i=1$, 
and $P_i^{\epsilon_i}$ is the conjugation of the inverse of the primitive relator $t_{c_{i-1}}^{-1}t_{b_i}t_{c_i}t_{b_i}^{-1}$ by $t_{b_i}^{-1}$ if $\epsilon_i=-1$ since $t_{b_i}(c_i)=c_{i-1}$ from $c_i=t_{b_i}^{-1}(c_{i-1})$. 
Here, we write
\begin{align*}
P_i^{\epsilon_i} &= V_it_{b_i}^{-\epsilon_i} 
\end{align*}
for $i=1,2,\ldots,k$, where $V_i = t_{c_i}^{-1}t_{b_i}^{\epsilon_i}t_{c_{i-1}}$. 
Note that 
\begin{align}
V_k \cdots V_2 V_1&= t_{c_k}^{-1} t_{b_k}^{\epsilon_k} \cdots t_{b_2}^{\epsilon_2} t_{b_1}^{\epsilon_1} t_{c_0} =t_{f(a)}^{-1}ft_a \label{V_1V_2V_k}
\end{align}
Here, let us consider the following conjugation $Q_i^{\epsilon_i}$ of $P_i^{\epsilon_i}$: 
\begin{align*}
Q_i^{\epsilon}&:= (V_k \cdots V_{i+2} V_{i+1}) P_i^{\epsilon_i} (V_k \cdots V_{i+2} V_{i+1})^{-1} \\
&= (V_k \cdots V_{i+2} V_{i+1} V_i) t_{b_i}^{-\epsilon_i} (V_k \cdots V_{i+2} V_{i+1})^{-1}. 
\end{align*}
Then, from (\ref{V_1V_2V_k}) we have 
\begin{align*}
Q_1^{\epsilon_1} Q_2^{\epsilon_2} \cdots Q_k^{\epsilon_k} = V_k \cdots V_2 V_1 t_{b_1}^{-\epsilon_1} t_{b_2}^{-\epsilon_2} \cdots t_{b_k}^{-\epsilon_k} = t_{f(a)}^{-1}ft_af^{-1}. 
\end{align*}
This finishes the proof. 
\end{proof}
From Lemma~\ref{lemma:???}, we can regard the word $t_{f(a)}^{-1}ft_af^{-1}$ as a primitive relator, so we use the same letter $P$ for $t_{f(a)}^{-1}ft_af^{-1}$, and we call the relation $ft_af^{-1}=t_{f(a)}$ the \textit{primitive braid relation} again. 
Moreover, the two well-known relations, called the commutative and the braid relations, are also the primitive braid relations. 
\begin{defn}\label{primitive}\rm
Let $a,b$ be two simple closed curves on $\Sigma_g^r$. 
\begin{itemize}
\item Let $f$ be an element in $\CM_g^r$. Then, we have the \textit{primitive braid relator} $ft_af^{-1}=t_{f(a)}$ and the \textit{primitive relator} 
\[P:=t_{f(a)}^{-1}ft_af^{-1}.\]

\item If $i(a,b)=0$, then $t_b(a)=a$. Therefore, we have the \textit{commutative relation} $t_a t_b = t_b t_a$ in $\CM_g^r$ and the \textit{commutative relator}
\[P:= t_a^{-1}t_bt_at_b^{-1},\] 

\item If $i(a,b)=1$, then $t_at_b(a)=b$. Then, the \textit{braid relation} $t_a t_b t_a = t_b t_a t_b$ holds in $\CM_g^r$. 
This gives the \textit{braid relator}
\[P:= t_b^{-1}t_at_bt_a t_b^{-1}t_a^{-1}.\]
\end{itemize}
\end{defn}

\subsection{A signature formula}\label{signature}
We now present the work of \cite{ehkt}. 
Since (\ref{monodromyLF}) is normally generated by $T,P,C_2,L$ from Theorem~\ref{Luo}, we can count the number of these four relators included in (\ref{monodromyLF}). 
This fact is the key to state the result in \cite{ehkt}. 
\begin{thm}[\cite{ehkt}, Proposition 2.9]\label{ehktprop0}
Let $n^{\pm}(R)$ be the number of a relator $R^{\pm1}$ included in the global monodromy of a genus-$g$ achiral Lefschetz fibration $\pi:X \to \Sigma_h$, where $R=T,P,C_2,L$. 
We set $n(R)=n^+(R)-n^-(R)$. 
Then, we have 
\begin{align*}
\sigma(X)=-n(T)-7n(C_2)+n(L). 
\end{align*}
\end{thm}
\begin{rem}
Originally, Proposition 2.9 in \cite{ehkt} is stated in terms of a graphical method, called the ``chart" description. 
\end{rem}

From Theorem~\ref{ehktprop0}, we notice that primitive braid relators are not needed for the computation of $\sigma(X)$. 
Equivalently, if we have an achiral Lefschetz fibration $\pi':X'\to \Sigma_h$ with the monodromy obtained by applying primitive braid relations to that of $\pi:X\to \Sigma_h$, then $\sigma(X)=\sigma(X^\prime)$ holds. 
For this reason, we introduce the following notation.
\begin{defn}\label{substitution}\rm
Let $P$ be a primitive braid relator in $\CM_g^r$. 
\begin{itemize}
\item Let $V$ and $V'$ be elements in $\CM_g^r$ with $V'V^{-1}=P^\epsilon$, where $\epsilon=\pm 1$. 
Set
\begin{align*}
W  &:= U_1 V U_2, \\
W' &:= U_1 V' U_2, 
\end{align*}
where $U_1$ and $U_2$ are elements in $\CM_g^r$. 
Then, we can construct $W'$ from $W$ using $P$ as follows: 
\begin{align*}
(U_1 P^\epsilon U_1^{-1}) W = (U_1 P^\epsilon U_1^{-1}) U_1 V U_2  = U_1 V' U_2 = W'.
\end{align*}
When $W'$ is obtained from $W$ by applying a sequence of the above operations (i.e. by using the primitive braid relations), we denote it by
\[W \e W'.\]

\item We say that \textit{$W$ can commute with $W'$ modulo $P$} if the next relation holds:
 \[W \cdot W' \e W' \cdot W.\] 

\item Let $W_1,W_2,\ldots,W_n$ be elements in $\CM_g^r$. If the relation 
\[W_1 W_2 \cdots W_{n-1}W_n \e W_n W_1 W_2 \cdots W_{n-1} \]
holds, then we call it a \textit{cyclic permutation}. 
\end{itemize}
\end{defn}

Remark~\ref{fundamental} below collects fundamental properties of Definition~\ref{substitution}. We will use it (without specifying) repeatedly. 
\begin{rem}\label{fundamental}
Let $f, X_1, X_2$ be elements in $\CM_g^r$, and let $a,a_1,a_2,\ldots,a_k$ be simple closed curves on $\Sigma_g^r$. 
We follow the notation of Definition~\ref{substitution}. 
\begin{enumerate}
\item[{\rm (1)}] For a primitive braid relator $P=t_{f(a)}^{-1}ft_af^{-1}$, we set $V=t_{f(a)}$, $V^\prime=ft_{a}f^{-1}$, $U_1 = X_1$, $U_2= X_2$. 
Then, we have
\[X_1 \cdot t_{f(a)} \cdot X_2 \e X_1 \cdot f t_a f^{-1} \cdot X_2.\]
\item[{\rm (2)}] For a primitive braid relator $P=t_{f(a)}^{-1}ft_af^{-1}$, we set $V=f$, $V^\prime=t_{f(a)}^{-1}ft_a$, $U_1 = t_{f(a)}$, $U_2= \mathrm{id}$. 
Then, we have
\[X_1 \cdot t_{f(a)} \cdot f \cdot X_2 \e X_1 \cdot f \cdot t_a \cdot X_2,\]
In particular, for any element $f$, the Dehn twist along a boundary curve $\partial$ of $\Sigma_g^r$ can commute with $f$ modulo $P$ from $f(\partial)=\partial$. 

\item[{\rm (3)}] A cyclic permutation always holds for a relator $R$ for the following reason: 
if we set $R=t_{a_1}^{\epsilon_1} t_{a_2}^{\epsilon_2} \cdots t_{a_k}^{\epsilon_k}$, where $\epsilon_i=\pm 1$, then $t_{a_k}^{-1}Rt_{a_k}R^{-1}$ is a primitive braid relator from $R(a_k)=a_k$. Therefore, for $P^{-1}=Rt_{a_k}^{-1}R^{-1}t_{a_k}$, when we set $V=R$, $V'=t_{a_k}Rt_{a_k}^{-1}$, $U_1=t_{a_k}$ and $U_2=\mathrm{id}$, we have 
\[t_{a_1}^{\epsilon_1} t_{a_2}^{\epsilon_2} \cdots t_{a_{k-1}}^{\epsilon_{k-1}} t_{a_k}^{\epsilon_k} \e t_{a_k}^{\epsilon_k} \cdot t_{a_1}^{\epsilon_1} t_{a_2}^{\epsilon_2} \cdots t_{a_{k-1}}^{\epsilon_{k-1}}\]
\end{enumerate}
\end{rem}

\section{Lemmas}\label{lemmas}
This section exhibits techniques to prove the main results.

From Section~\ref{section:2}, we see that we need to write relators as a product of commutators. 
The next lemma will be useful for constructing commutators. 
The special cases were used in \cite{harer}, \cite{ko1} and \cite{bkm}. 
\begin{lem}\label{lemma:67}
Let $a_1,a_2,\ldots,a_n$ and $b_1,b_2,\ldots,b_n$ be simple closed curves on $\Sigma_g^r$. If there is an element $f$ in $\CM_g^r$ mapping $(a_1,a_2,\ldots,a_n)$ to $(b_1,b_2,\ldots,b_n)$, then 
for any integers $k_1,k_2,\ldots,k_n$, the following holds:
\[t_{a_1}^{k_1} t_{a_2}^{k_2} \cdots t_{a_n}^{k_n} \ t_{b_n}^{-k_n} \cdots t_{b_2}^{-k_2} t_{b_1}^{-k_1} \e [t_{a_1}^{k_1} t_{a_2}^{k_2} \cdots t_{a_n}^{k_n}, f].\]
\end{lem}
\begin{proof}
By the primitive braid relations and $(ft_{a_i}f^{-1})^{-k_i} = ft_{a_i}^{-k_i}f^{-1}$, we have
\begin{align*}
t_{a_1}^{k_1} t_{a_2}^{k_2} \cdots t_{a_n}^{k_n} \ t_{b_n}^{-k_n} \cdots t_{b_2}^{-k_2} t_{b_1}^{-k_1} &= t_{a_1}^{k_1} t_{a_2}^{k_2} \cdots t_{a_n}^{k_n} \cdot t_{f(a_n)}^{-k_n} \cdots t_{f(a_2)}^{-k_2} t_{f(a_1)}^{-k_1} \\
&\e t_{a_1}^{k_1} t_{a_2}^{k_2} \cdots  t_{a_n}^{k_n} f t_{a_n}^{-k_n} \cdots t_{a_2}^{-k_2} t_{a_1}^{-k_1} f^{-1}. 
\end{align*}
By $t_{a_n}^{-k_n} \cdots t_{a_2}^{-k_2} t_{a_1}^{-k_1} = (t_{a_1}^{k_1} t_{a_2}^{k_2} \cdots  t_{a_n}^{k_n})^{-1}$, we obtain the required formula. 
\end{proof}

The next three lemmas are used to construct an element $f$ in Lemma~\ref{lemma:67}. 
\begin{lem}\label{lemma:10001}
Let $a,b,c$ be nonseparating curves on $\Sigma_g^r$ such that $i(a,b)=i(b,c)=1$. 
Then the following holds. 
\begin{enumerate}
\item[{\rm (1)}] $t_bt_ct_at_b$ maps $a$ to $c$. It maps $(a,c)$ to $(c,a)$ if $i(a,c)=0$, 
\item[{\rm (2)}] $t_at_bt_c$ maps $(a,b)$ to $(b,c)$ if $i(a,c)=0$. 
\end{enumerate}
\end{lem}
\begin{proof}
Since $t_at_b(a)=b$, $t_bt_c(b)=c$, $t_ct_b(c)=b$ and $t_bt_a(b)=a$, and $t_a(c)=c$, $t_c(a)=a$ and $t_at_c=t_ct_a$ if $i(a,c)=0$ (see Definition~\ref{primitive}), (1) follows from 
\begin{align*}
t_bt_ct_at_b(a) &= t_bt_c(b) = c, \\
t_bt_ct_at_b(c) &= t_bt_at_ct_b(c) = t_bt_a(b) = a, 
\end{align*}
and (2) is obtained as follows:
\begin{align*}
t_at_bt_c(a) &= t_at_b(a) = b, \\
t_at_bt_c(b) &= t_a(c) = c. 
\end{align*}
\end{proof}

\begin{lem}\label{lemma:10002}
Let $a,b,c,\alpha,\beta,\gamma$ be nonseparating curves on $\Sigma_g^r$ such that $i(a,b)=i(b,c)=i(\alpha,\beta)=i(\beta,\gamma)=1$. 
Suppose that $a$ (resp. $\gamma$) is disjoint from $\alpha,\beta,\gamma$ (resp. $a,b,c$). 
Then, $t_{b}t_{c}t_{a}t_{b} \cdot t_{\beta}t_{\gamma}t_{\alpha}t_{\beta}$ maps $(a,\alpha)$ to $(c,\gamma)$. 
It maps $(a,c,\alpha,\gamma)$ to $(c,a,\gamma,\alpha)$ if $c$ (resp. $\alpha$) is disjoint from $a, \alpha,\beta,\gamma$ (resp. $\gamma, a,b,c$). 
\end{lem}
\begin{proof}
Since $a$ (resp. $\gamma$) is disjoint from $\alpha,\beta,\gamma$ (resp. $a,b,c$), we have 
\begin{align*}
t_b t_c t_a t_b \cdot t_{\beta} t_{\gamma} t_{\alpha} t_{\beta} (a) &= t_b t_c t_a t_b (a) = c, \\
t_b t_c t_a t_b \cdot t_{\beta} t_{\gamma} t_{\alpha} t_{\beta} (\alpha) &= t_b t_c t_a t_b (\gamma)  = \gamma 
\end{align*}
by the farmer part of Lemma~\ref{lemma:10001} (1). By a similar argument, the latter part of Lemma~\ref{lemma:10002} follows from that of Lemma~\ref{lemma:10001} (1).  
This finishes the proof. 
\end{proof}

\begin{lem}\label{lemma:10003}
Let $a,b,c,\alpha,\beta,\gamma$ be nonseparating curves on $\Sigma_g^r$ such that $i(a,b)=i(b,c)=i(\alpha,\beta)=i(\beta,\gamma)=1$. 
Suppose that $a,c$ (resp. $\beta$) are disjoint from $\alpha,\beta,\gamma$ (resp. $a,b,c$). 
Then, $ t_{\beta}t_{\gamma} \cdot t_{b}t_{c}t_{a}t_{b} \cdot t_{\alpha}t_{\beta}$ maps $(a,\alpha)$ to $(c,\gamma)$. 
\end{lem}
\begin{proof}
Since $a,c$ (resp. $\beta$) are disjoint from $\alpha,\beta,\gamma$ (resp. $a,b,c$), by $t_\alpha t_\beta (\alpha) = \beta$, $t_\beta t_\gamma (\beta) = \gamma$ (see Definition~\ref{primitive}) and the farmer part of Lemma~\ref{lemma:10001} (1), we have
\begin{align*}
t_{\beta} t_{\gamma} \cdot t_b t_c t_a t_b \cdot t_{\alpha} t_{\beta} (a) &= t_{\beta} t_{\gamma} \cdot t_b t_c t_a t_b (a) = t_{\beta} t_{\gamma} (c) = c, \\
t_{\beta} t_{\gamma} \cdot t_b t_c t_a t_b \cdot t_{\alpha} t_{\beta} (\alpha) &= t_{\beta} t_{\gamma} \cdot t_b t_c t_a t_b (\beta)  = t_{\beta} t_{\gamma} (\beta) = \gamma,
\end{align*}
and this finishes the proof. 
\end{proof}

The key lemma of this paper is following. 
\begin{lem}\label{lemma:8}
Let $a_1,a_2,\ldots,a_{m+1}$ be disjoint simple closed curves on $\Sigma_g^r$. 
If there is an element $f$ in $\CM_g^r$ such that $f(a_i)=a_{i+1}$ for $i=1,2,\ldots,m$, then we have the following relations in $\CM_g^r$ for any integers $k_1,k_2,\ldots,k_{m+1}$: 
\begin{enumerate}
\item[{\rm (1)}] $t_{a_1}^{k_1} t_{a_2}^{k_2} \cdots t_{a_{m+1}}^{k_{m+1}} \e [t_{a_1}^{k_1} t_{a_2}^{k_1+k_2} \cdots t_{a_m}^{k_1+k_2+\cdots + k_m}, f] \cdot t_{a_{m+1}}^{k_1+k_2+\cdots +k_{m+1}}$, 
\item[{\rm (2)}] $t_{a_1}^{k_1} t_{a_2}^{k_2} \cdots t_{a_{m+1}}^{k_{m+1}} \e t_{a_{m+1}}^{k_1+k_2+\cdots +k_{m+1}} \cdot [t_{a_1}^{k_1} t_{a_2}^{k_1+k_2} \cdots t_{a_m}^{k_1+k_2+\cdots + k_m}, f]$, 
\item[{\rm (3)}] $t_{a_1}^{k_1} t_{a_2}^{k_2} \cdots t_{a_{m+1}}^{k_{m+1}} \e t_{a_{m+1}}^{k_1+k_2+\cdots +k_{m+1}} \cdot [f, t_{a_1}^{-k_1} t_{a_2}^{-k_1-k_2} \cdots t_{a_m}^{-k_1-k_2-\cdots-k_m}]$.
\end{enumerate}
\end{lem}
\begin{proof}
For abbreviation, set $K_i:=k_1+k_2+\cdots + k_i$. 
Then, we have
\[t_{a_1}^{k_1} t_{a_2}^{k_2} \cdots t_{a_{m+1}}^{k_{m+1}} = t_{a_1}^{K_1} t_{a_2}^{-K_1} \cdot t_{a_2}^{K_2} t_{a_3}^{-K_2} \cdot t_{a_3}^{K_3} t_{a_4}^{-K_3} \cdots t_{a_m}^{K_m} t_{a_{m+1}}^{-K_m} \cdot t_{a_{m+1}}^{K_{m+1}}.\] 
This relation and the commutative relations give the following three relations:
\begin{align*}
t_{a_1}^{k_1} t_{a_2}^{k_2} \cdots t_{a_{m+1}}^{k_{m+1}} &\e t_{a_1}^{K_1} t_{a_2}^{K_2} \cdots t_{a_m}^{K_m} \cdot t_{a_{m+1}}^{-K_m} \cdots t_{a_3}^{-K_2} t_{a_2}^{-K_1} \cdot t_{a_{m+1}}^{K_{m+1}}, \\
t_{a_1}^{k_1} t_{a_2}^{k_2} \cdots t_{a_{m+1}}^{k_{m+1}} &\e t_{a_{m+1}}^{K_{m+1}} \cdot t_{a_1}^{K_1} t_{a_2}^{K_2} \cdots t_{a_m}^{K_m} \cdot t_{a_{m+1}}^{-K_m} \cdots t_{a_3}^{-K_2} t_{a_2}^{-K_1}, \\
t_{a_1}^{k_1} t_{a_2}^{k_2} \cdots t_{a_{m+1}}^{k_{m+1}} &\e t_{a_{m+1}}^{K_{m+1}} \cdot t_{a_2}^{-K_1} t_{a_3}^{-K_2} \cdots t_{a_{m+1}}^{-K_m} \cdot t_{a_m}^{K_m} \cdots t_{a_2}^{K_2} t_{a_1}^{K_1}. 
\end{align*}
Here, by the primitive braid relation $t_{a_{i+1}} \e ft_{a_i}f^{-1}$ and $(ft_{a_i}f^{-1})^{-K_i} = ft_{a_i}^{-K_i}f^{-1}$ for $i=1,2,\ldots,m$, we obtain 
\begin{align*}
t_{a_1}^{K_1} t_{a_2}^{K_2} \cdots t_{a_m}^{K_m} \cdot t_{a_{m+1}}^{-K_m} \cdots t_{a_3}^{-K_2} t_{a_2}^{-K_1} &\e t_{a_1}^{K_1} t_{a_2}^{K_2} \cdots t_{a_m}^{K_m} \cdot f t_{a_m}^{-K_m} \cdots t_{a_2}^{-K_2} t_{a_1}^{-K_1} f^{-1}, \\
t_{a_2}^{-K_1} t_{a_3}^{-K_2} \cdots t_{a_{m+1}}^{-K_m} \cdot t_{a_m}^{K_m} \cdots t_{a_2}^{K_2} t_{a_1}^{K_1} &\e f t_{a_1}^{-K_1} t_{a_2}^{-K_2} \cdots t_{a_m}^{-K_m} f^{-1} \cdot t_{a_m}^{K_m} \cdots t_{a_2}^{K_2} t_{a_1}^{K_1}. 
\end{align*}
Hence, the relations (1)--(3) follow from $t_{a_m}^{-K_m} \cdots t_{a_2}^{-K_2} t_{a_1}^{-K_1} = (t_{a_1}^{K_1} t_{a_2}^{K_2} \cdots t_{a_m}^{K_m})^{-1}$ and $t_{a_m}^{K_m} \cdots t_{a_2}^{K_2} t_{a_1}^{K_1} = (t_{a_1}^{-K_1} t_{a_2}^{-K_2} \cdots t_{a_m}^{-K_m})^{-1}$. 
\end{proof}

The next four lemmas are used to reduce the number of commutators. 
\begin{lem}\label{lemma:10}
For elements $X_1,X_2,Y_1,Y_2$ in $\CM_g^r$ with $X_iY_j \e Y_jX_i$, we have
\[[X_1,X_2][Y_1,Y_2] \e [X_1Y_1,X_2Y_2]. \] 
\end{lem}
\begin{proof}
It follows from 
\begin{align*}
X_1X_2X_1^{-1}X_2^{-1} \cdot Y_1Y_2Y_1^{-1}Y_2^{-1} \e X_1Y_1 X_2Y_2 Y_1^{-1}X_1^{-1} Y_2^{-1}X_2^{-1}. 
\end{align*}
\end{proof}

\begin{lem}\label{XYYZ}
For any three elements $X,Y,Z$ in a group $G$, we have
\[[X,Y][Y,Z] = [XZ^{-1},ZYZ^{-1}].\]
\end{lem}
\begin{proof}
The equation immediately follows from the following computations$:$
\begin{align*}
[X,Y][Y,Z] &= XYX^{-1}Y^{-1} \cdot YZY^{-1}Z^{-1} = XYX^{-1} ZY^{-1}Z^{-1}, \\
[XZ^{-1},ZYZ^{-1}] &= (XZ^{-1}) (ZYZ^{-1}) (ZX^{-1}) (ZY^{-1}Z^{-1}) = XYX^{-1} ZY^{-1}Z^{-1}. 
\end{align*}
\end{proof}

\begin{lem}\label{lemma:31}
Let $X,Y$ be element in $\CM_g^r$. For any integer $n$, we have
\begin{enumerate}
\item[{\rm (1)}] $\displaystyle (XY)^n = {}_{X}(Y) {}_{X^2}(Y) \cdots {}_{X^n}(Y) \cdot X^n$, 
\item[{\rm (2)}] $\displaystyle (XY)^n = X^n \cdot {}_{X^{-n+1}}(Y) \cdots {}_{X^{-2}}(Y) {}_{X^{-1}}(Y) Y$. 
\end{enumerate}
\end{lem}
\begin{proof}
The equations immediately follow from 
\begin{align*}
(XY)^n &= (XYX^{-1})(X^2YX^{-2}) \cdots (X^nYX^{-n}) X^n, \\
(XY)^n &= X^n (X^{-n+1}YX^{n-1}) \cdots (X^{-2}YX^2) (X^{-1}YX) Y. 
\end{align*}
\end{proof}

\begin{lem}\label{lemma:21}
Let $X$ and $f$ be elements in $\CM_g^r$ such that $X$ is the product $X=X_1X_2\cdots X_n$ satisfying $X_i \cdot X_j \e X_j \cdot X_i$ for $i\neq j$, $X_{i+1} \cdot f \e f \cdot X_i $ and $X_1 \cdot f \e f \cdot X_n $.  Then, we have
\[X \cdot f \e f \cdot X.\]
\end{lem}
\begin{proof}
We obtain the claim as follows:
\begin{align*}
X_1 X_2 \cdots X_{n-1} X_n \cdot f &\e X_2 X_3 \cdots X_n X_1 \cdot f \e f\cdot X_1 X_2 \cdots X_{n-1} X_n. 
\end{align*}
\end{proof}

\section{Scl of the Dehn twist along a nonseparating curve}\label{section:two}
We first give the proof of Theorem~\ref{thmD} (1) since some results in it will be used in the proofs of Theorem~\ref{thmA}, \ref{thmB} and \ref{thmE}. 
Note that since Dehn twists along two nonseparating curves $s_0,s_0'$ (resp. two separating curves $s_h,s_h'$ of type $h$ and a separating curve $s_{g-h}$ of type $g-h$) are conjugate, and a conjugate of a commutator is again a commutator, it suffices to prove Theorem~\ref{thmD} and~\ref{thmE} for some nonseparating curve (resp. separating curve of type $h$).

In order to prove Theorem~\ref{thmD} (1), we present the 3-chain relator and factorize its $n$-the power as a product of commutators and Dehn twists. 
The factorization will be used to show Theorem~\ref{thmA}, \ref{thmC}, \ref{thmD} (1) and \ref{thmE} (1) and (2). 
\begin{defn}\label{3-chain}\rm 
Let $a,b,c$ be simple closed curves on $\Sigma_1^2$ bounded by $d,d'$ with $i(a,b)=i(b,c)=1$ and $i(c,a)=0$ as in Figure~\ref{abcdd'}. 
Then, we have the \textit{3-chain relation} $t_{d^\prime}t_d=(t_at_bt_c)^4$ in $\CM_1^2$ and the \textit{3-chain relator}
\[C_3 :=  t_d^{-1}t_{d^\prime}^{-1}(t_at_bt_c)^4.\]
\end{defn}
\begin{figure}[hbt]
  \centering
       \includegraphics[scale=.95]{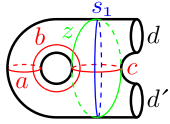}
       \caption{The curves $a,b,c,d,d',s_1,z$ on $\Sigma_1^2$.}
       \label{abcdd'}
  \end{figure}

The next proposition is the key result in Section~\ref{section:two}. 
We will use some equations in the proof to show Theorem~\ref{thmA}, \ref{thmC}, \ref{thmD} (1) and \ref{thmE} (1) and (2). 
\begin{prop}\label{prop:2}
In the notation of Definition~\ref{3-chain}, for any integer $n$, there are elements $V_1,W_1,V_2,W_2,\ldots,V_{|n|+1},W_{|n|+1}$ in $\CM_1^2$ such that the following holds in $\CM_1^2$:
\[C_3^n \e t_{b}^{12n} [V_1,W_1] [V_2,W_2] \cdots [V_{|n|+1},W_{|n|+1}] \cdot t_{d}^{-n} t_{d'}^{-n}. \]
\end{prop}
\begin{proof}
Let $v=t_{a}t_{c}(b)$. Since $a$ is disjoint from $c$, $t_c^{-1}t_a^{-1}(c)=c$ and $t_c^{-1}t_a^{-1}(a)=a$ (see Definition~\ref{primitive}). 
By the primitive braid relation and Lemma~\ref{lemma:10001} (2), we have 
\begin{align*}
t_{b} t_{v}(a) = t_{b} t_{a}t_{c}t_{b}t_{c}^{-1}t_{a}^{-1} (a) = t_{b} t_{a}t_{c}t_{b} (a) = c, \\
t_{b} t_{v}(c) = t_{b} t_{a}t_{c}t_{b}t_{c}^{-1}t_{a}^{-1} (c) = t_{b} t_{a}t_{c}t_{b} (c) = a.
\end{align*}
This gives the following two relations:
\begin{align}
t_{b}t_{v} \cdot t_{a}&\e t_{c} \cdot t_{b}t_{v}, \label{bva}\\
t_{b}t_{v} \cdot t_{c}&\e t_{a} \cdot t_{b}t_{v}. \label{bvc}
\end{align}
Note that using the primitive braid relation, we have
\begin{align*}
t_{a} t_{b} t_{c} t_{a} t_{b} t_{c} &\e t_{a} t_{b} t_{a} t_{c} t_{b} t_{c}  \\
&\e t_{a} \cdot t_{b} (t_{a} t_{c} t_{b} t_{c}^{-1} t_{a}^{-1}) \cdot t_{a} t_{c} t_{c}  \\
&\e t_{a} \cdot t_{b} t_{v} \cdot t_{a} t_{c} t_{c}. 
\end{align*}
This equation, the relations (\ref{bva}) and (\ref{bvc}), the commutative relation $t_{a}t_{c}=t_{c}t_{a}$ and a cyclic permutation give
\begin{align*}
C_3 \e t_{a}^4 t_{c}^4 (t_{b} t_{v})^2 t_{d}^{-1}t_{d'}^{-1}. 
\end{align*}
When we take $n$-th power of this relation, by the property of boundary curves $d,d^\prime$, the relations (\ref{bva}) and (\ref{bvc}) and the commutative relation $t_{a}t_{c}=t_{c}t_{a}$, we have 
\begin{align}\label{relation10000}
C_3^n \e t_{a}^{4n} t_{c}^{4n} (t_{b} t_{v})^{2n} t_{d}^{-n} t_{d'}^{-n}. 
\end{align}
By this equation and the primitive braid relations, we have
\begin{align*}
C_3^n &\e t_{a}^{4n} t_{c}^{4n} (t_{b}^4 \cdot t_{b}^{-1} (t_{b}^{-2} t_{v} t_{b}^2) t_{b}^{-1} t_{v})^n t_{d}^{-n} t_{d'}^{-n}  \\
&\e t_{a}^{4n} t_{c}^{4n} (t_{b}^4 \cdot t_{b}^{-1} t_{t_{b}^{-2}(v)} t_{b}^{-1} t_{v})^n t_{d}^{-n} t_{d'}^{-n} . 
\end{align*}
Here, when we set $\phi_3 := t_{a} t_{c} t_{b}^3$ in $\CM_1^2$, $\phi_3(b) = t_{a} t_{c}(b) = v$ and $\phi_3(t_{b}^{-2}(v)) = t_{a} t_{c} t_{b}(v) = t_{a} t_{c} t_{b} t_{a} t_{c}(b)$. 
From the commutative and the braid relations, we have 
\begin{align*}
t_{a} t_{c} t_{b} t_{a} t_{c} = t_{a} t_{c} t_{b} t_{c} t_{a} = t_{a} t_{b} t_{c} t_{b} t_{a}. 
\end{align*}
By Lemma~\ref{lemma:10001} (2), we see that
\begin{align*}
t_{a} t_{b} t_{c} t_{b} t_{a}(b) = t_{a} t_{b} t_{c}(a) = b, 
\end{align*}
so $\phi_3(t_{b}^{-2}(v))=b$. 
Therefore, $\phi_3$ maps $(b,t_{b}^{-2}(v))$ to $(v,b)$. 
This gives 
\begin{align*}
C_3^n \e t_{a}^{4n} t_{c}^{4n} (t_{b}^4 \cdot [t_{b}^{-1} t_{t_{b}^{-2}(v)},\phi_3])^n t_{d}^{-n} t_{d'}^{-n}
\end{align*}
from Lemma~\ref{lemma:67}. 
Therefore, by Lemma~\ref{lemma:31} (2), we obtain the following relation:
\begin{align}\label{relation100}
C_3^n \e t_{a}^{4n} t_{c}^{4n} t_{b}^{4n} \cdot \prod_{i=1}^n {}_{t_{b}^{-4(i-1)}} ([t_{b}^{-1} t_{t_{b}^{-2}(v)}, \phi_3])  \cdot t_{d}^{-n} t_{d'}^{-n}. 
\end{align}
Note that the conjugation of a commutator is also a commutator, and that we have 
\begin{align*}
t_{a}^{4n} t_{c}^{4n} t_{b}^{4n} &\e t_{b}^{12n} \cdot t_{b}^{-4n} (t_{b}^{-8n} t_{a}^{4n} t_{b}^{8n}) t_{b}^{-4n} (t_{b}^{-4n} t_{c}^{4n} t_{b}^{4n})  \\
&\e t_{b}^{12n} \cdot t_{b}^{-4n} t_{t_{b}^{-8n}(a)}^{4n} t_{b}^{-4n} t_{t_{b}^{-4n}(c)}^{4n}. 
\end{align*}
Since $t_at_bt_c$ maps $(a,b)$ to $(b,c)$ by Lemma~\ref{lemma:10001} (2), we find that $t_b^{-4n}t_at_bt_ct_b^{8n}$, denoted $\phi_4$, maps $(b,t_{b}^{-8n}(a))$ to $(t_{b}^{-4n}(c),b)$, so Lemma~\ref{lemma:67} gives 
\[t_{a}^{4n} t_{c}^{4n} t_{b}^{4n} \e t_{b}^{12n} \cdot [t_{b}^{-4n} t_{t_{b}^{-8n}(a)}^{4n}, \phi_4], \] 
and this establishes the formula.
\end{proof}

Theorem~\ref{thmD} (1) directly follows from Theorem~\ref{thm:1000} below, which will be used to prove Theorem~\ref{thmB}, since the left hand side of the equation in it is relator. 
\begin{thm}\label{thm:1000}
Let $s_0$ be a nonseparating curve on $\Sigma_g$ for $g\geq 2$ and $C_{3,j}$ a $3$-chain relator. 
For any integer $n$, there are elements $\mathcal{V}_1,\mathcal{W}_1,\mathcal{V}_2,\mathcal{W}_2,\ldots,\mathcal{V}_{|n|+3},\mathcal{W}_{|n|+3}$ in $\CM_g$ such that 
\begin{align*}
\prod_{j=1}^{g-1} C_{3,j}^n \e  t_{s_0}^{10(g-1)n} [\mathcal{V}_1,\mathcal{W}_1] [\mathcal{V}_2,\mathcal{W}_2] \cdots [\mathcal{V}_{|n|+3},\mathcal{W}_{|n|+3}]. 
\end{align*}
\end{thm}
\begin{proof}
Let us consider the simple closed curves $a_1,b_1,c_1$ on the genus-1 subsurface $S_1^2$ of $\Sigma_g$ bounded by $d_1,d_{g-1}$ as in Figure~\ref{Li}. 
Then, we obtain the 3-chain relator $C_{3,1}:=t_{d_{g-1}}^{-1}t_{d_1}^{-1}(t_{a_1}t_{b_1}t_{c_1})^4$. 
By Proposition~\ref{prop:2}, the relation
\begin{align*}
C_{3,1}^n \e t_{b_1}^{10n}[V_{1,1},W_{1,1}] [V_{2,1},W_{2,1}] \cdots [V_{|n|+1,1},W_{|n|+1},1] t_{d_{g-1}}^{-n} t_{d_1}^{-n},
\end{align*}
holds in $\CM(S_1^2)$ for any integer $n$, where $V_{i,1},W_{i,1}$ are some elements in $\CM(S_1^2)$. 

  \begin{figure}[hbt]
  \centering
       \includegraphics[scale=.80]{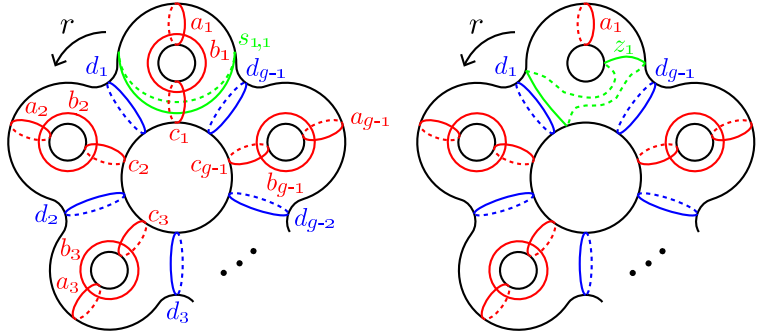}
       \caption{The rotation $r$ of $\Sigma_g$ and the curves $a_1,b_1,c_1,d_1,d_{g-1},s_{1,1},z_1$.}
       \label{Li}
  \end{figure}
Let $r$ be the rotation of $\Sigma_g$ by $2\pi/(g-1)$ as in Figure~\ref{Li}. 
We set
\begin{align*}
&C_{3,j} := {}_{r^{j-1}}(C_{3,1}),& && \\
&b_j := r^{j-1}(b_1),& &d_j := r^{j-1}(d_1),& \\
&V_{i,j} := {}_{r^{j-1}}(V_i)& &W_{i,j} := {}_{r^{j-1}}(W_i)& 
\end{align*}
for $j=1,2,\ldots,g-1$. 
Note that $b_g = b_1$, $d_g=d_1$. 
Then, using the primitive braid relations, the relation holds in $\CM(r^{j-1}(S_1^2))$:
\[ C_{3,j}^n \e t_{b_j}^{10n}[V_{1,j},W_{1,j}] [V_{2,j},W_{2,j}] \cdots [V_{|n|+1,j},W_{|n|+1},j] t_{d_j}^{-n} t_{d_{j+1}}^{-n} \]
for $j=1,2,\ldots,g-1$. 
Here, any simple closed curves on $\mathrm{Int}(r^{j-1}(S_1^2))$ are disjoint from any simple closed curves on $\mathrm{Int}(r^{j^\prime-1}(S_1^2))$ if $j\neq j^\prime$, and $d_j,d_{j+1}$ are boundary curves of of $r^{j-1}(S_1^2)$. 
Hence, for any elements $e_j$ in $\CM(r^{j-1}(S_1^2))$ and any element $f_{j^\prime}$ in $\CM(r^{j^\prime-1}(S_1^2))$, we have $e_j f_{j^\prime} = f_{j^\prime} e_j$ by the commutative relations and the property of boundary curves if $j\neq j^\prime$. 
From Lemma~\ref{lemma:10} and $d_g=d_1$, we have 
\begin{align*}
\prod_{j=1}^{g-1} C_{3,j}^n &\e  \prod_{j=1}^{g-1} t_{b_j}^{10n} \cdot \prod_{i=1}^{|n|+1}[\mathcal{V}_i,\mathcal{W}_i] \cdot \prod_{j=1}^{g-1} t_{d_j}^{-2n} \\
&\e  \prod_{j=1}^{g-1} t_{b_j}^{10n}  \cdot \prod_{j=1}^{g-1} t_{d_j}^{-2n} \cdot \prod_{i=1}^{|n|+1}[\mathcal{V}_i,\mathcal{W}_i]
\end{align*}
where $\mathcal{V}_i = V_{i,1} V_{i,2} \cdots V_{i,g-1}$ and $\mathcal{W}_i = W_{i,1} W_{i,2} \cdots W_{i,g-1}$. 
Using Lemma~\ref{lemma:8} (2) and (3), we see that
\begin{align*}
&\prod_{j=1}^{g-1}t_{b_j}^{12n} \e t_{b_{g-1}}^{12(g-1)n} [B,r], \\
&\prod_{j=1}^{g-1} t_{d_j}^{-2n} \e  t_{d_{g-1}}^{-2(g-1)n} [r,D],
\end{align*}
where $B := t_{b_1}^{12n} t_{b_2}^{24n} \cdots t_{b_{g-2}}^{12(g-2)n}$ and $D := t_{d_1}^{2n} t_{d_2}^{4n} \cdots t_{d_{g-2}}^{2(g-2)n}$. 
This gives 
\begin{align*}
\prod_{j=1}^{g-1}C_{3,j}^n \e t_{b_{g-1}}^{12(g-1)n} [B,r]  \cdot t_{d_{g-1}}^{-2(g-1)n} [r,D] \cdot \prod_{i=1}^{|n|+1}[\mathcal{V}_i,\mathcal{W}_i],
\end{align*}
Since $b_j,a_j$ are disjoint from each $d_g$ for any $j$, we have $B(d_g)=d_g$. 
This gives $[B,r](d_{g-1})=BrB^{-1}r^{-1}(d_{g-1})=d_{g-1}$, so we have 
$[B,r] t_{d_{g-1}}^{12(g-1)n} \e t_{d_{g-1}}^{12(g-1)n} [B,r]$. 
From this and Lemma~\ref{XYYZ}, we obtain 
\begin{align*}
\prod_{j=1}^{g-1}C_{3,j}^n \e t_{b_{g-1}}^{12(g-1)n} t_{d_{g-1}}^{-2(g-1)n} \cdot [BD^{-1},DrD^{-1}] \cdot \prod_{i=1}^{|n|+1}[\mathcal{V}_i,\mathcal{W}_i].
\end{align*}
Since $b_{g-1}$ and $d_{g-1}$ are nonseparating, there exists a diffeomorphism $f$ satisfying $f(b_{g-1})=d_{g-1}$. 
Therefore, by Lemma~\ref{lemma:67} we have 
\begin{align*}
t_{b_{g-1}}^{12(g-1)n} t_{d_{g-1}}^{-2(g-1)n} &= t_{b_{g-1}}^{10(g-1)n} \cdot t_{b_{g-1}}^{2(g-1)n} t_{d_{g-1}}^{-2(g-1)n} \\
&\e t_{b_{g-1}}^{10(g-1)n} [t_{b_{g-1}}^{2(g-1)n},f], 
\end{align*}
and this proves Theorem~\ref{thm:1000} and therefore Theorem~\ref{thmD} (1). 
\end{proof}
\begin{rem}
M. Korkmaz spoke an interesting proof giving an upper bound on $\scl_{\CM_g}(t_{s_0})$ at Max Plank, 2013 (see \cite{MaxPlank}). 
The main idea is to use his result of \cite{ko} and quasi-morphisms and to consider $[\frac{g}{2}]$ disjoint genus-2 subsurfaces of $\Sigma_g$ with one boundary. 
The proof of Theorem~\ref{thmD} is much inspired by his idea. 
\end{rem}

\section{Surface bundles with base of genus two}\label{Surface bundles with base of genus two}
In this section, we prove Theorem~A.

Throughout this section, we suppose that $g\geq 39$. 
Let us consider $\Sigma_g^1$ with one boundary component $\partial$ as in Figure~\ref{surfaceS}. 
Then, we can take $13$ disjoint subsurfaces $S_1,S_2,\ldots,S_{12}$ and $S$ of genus $3$ with one boundary component and an element $\Phi$ in $\CM_g^1$ such that $\Phi(S_i)=S_{i+1}$, $\Phi(S_{12})=S_1$ and $\Phi|_S=\mathrm{id}|_S$ as in Figure~\ref{surfaceS}. 
  \begin{figure}[hbt]
  \centering
       \includegraphics[scale=.80]{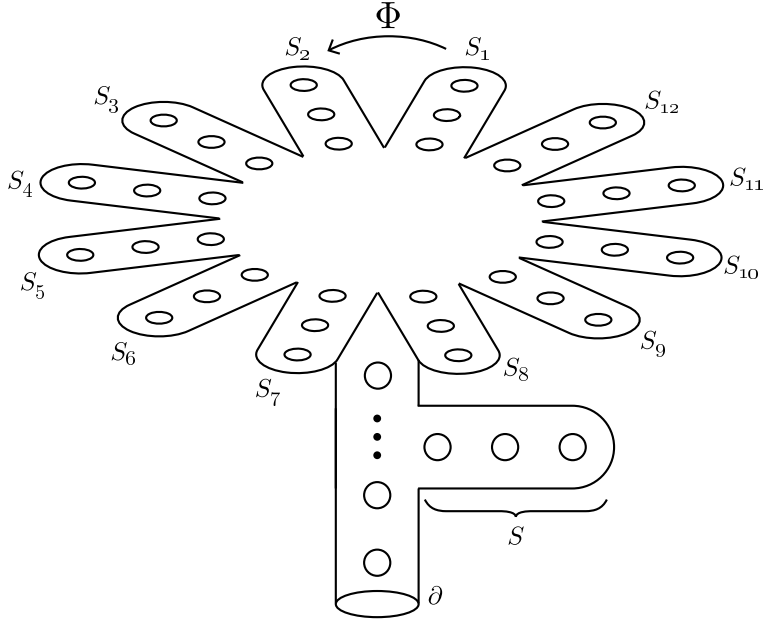}
       \caption{The subsurfaces $S,S_1,S_2,\ldots,S_{12}$ of $\Sigma_g^1$.}
       \label{surfaceS}
  \end{figure}

Let $\alpha_1,\beta_1,\gamma_1,\delta_1,\epsilon_1,\zeta_1,x_1,y_1,z_1$ be the simple closed curves on $S_1$ as Figure~\ref{L0L12}, and let $a_1,b_1,s_{1,1},d_1,d_2,\gamma, \delta, y, z, \epsilon, \zeta$ be simple closed curves on $S$ as in the figure. 
  \begin{figure}[hbt]
  \centering
       \includegraphics[scale=.90]{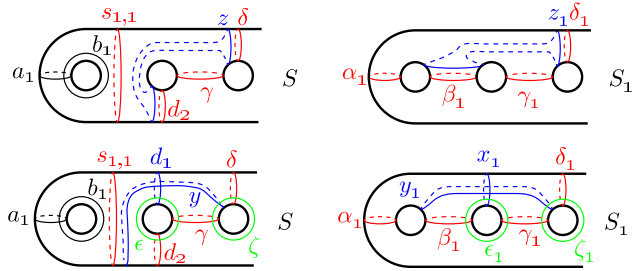}
       \caption{The curves $a_1,b_1,s_{1,1},d_1,d_2,\gamma, \delta, \epsilon, \zeta, y, z$ on $S$ and the curves $\alpha_1,\beta_1,\gamma_1,\delta_1,\epsilon_1,\zeta_1,x_1,y_1,z_1$ on $S_1$.}
       \label{L0L12}
  \end{figure}
We consider the following two lantern relators $L_1$ and $L$: 
\begin{align*}
L_1 &:= t_{\alpha_1}^{-1} t_{\delta_1}^{-1} t_{\gamma_1}^{-1} t_{\beta_1}^{-1} t_{x_1} t_{y_1} t_{z_1} , \\
L&:= t_{d_1} t_y t_z t_\delta^{-1} t_\gamma^{-1} t_{d_2}^{-1} t_{s_{1,1}}^{-1}. 
\end{align*}
The next lemma was proved in \cite{ko1}. 
\begin{lem}[\cite{ko1}]\label{example}
Set $\psi_1= t_{\epsilon_1} t_{\gamma_1} t_{x_1} t_{\epsilon_1}$, $\omega_1= t_{\epsilon_1} t_{z_1} t_{\beta_1} t_{\epsilon_1} t_{\zeta_1} t_{\delta_1} t_{y_1} t_{\zeta_1}$, $\psi= t_\epsilon t_\gamma t_{d_1} t_\epsilon$ and $\phi=t_\zeta t_y t_\delta t_\zeta t_\epsilon t_{d_2} t_z t_\epsilon$. 
The followings hold in $\CM(S_1)$ and $\CM(S)$, respectively:
\begin{align*}
L_1 &\e [t_{x_1},\psi_1] \cdot [t_{y_1} t_{\beta_1}^{-1},\omega_1] \cdot t_{\alpha_1}^{-1}, \\ 
L^{-1} &\e t_{s_{1,1}} \cdot [t_\delta t_z^{-1},\phi] \cdot [t_\gamma,\psi], 
\end{align*}
\end{lem}
\begin{proof}
Since $\alpha_1,\beta_1,\gamma_1,\delta_1$ (resp. $\delta,\gamma,d_2,s_{1,1}$) are disjoint from $x_1,y_1,z_1$ (resp. $d_1,y,z$) and disjoint from each other, the commutative relations give
\begin{align*}
L_1 &\e t_{x_1} t_{\gamma_1}^{-1} \cdot t_{y_1} t_{\beta_1}^{-1} t_{z_1} t_{\delta_1}^{-1} \cdot t_{\alpha_1}^{-1} \\
L^{-1} &\e t_{s_{1,1}} \cdot t_\delta t_z^{-1} t_{d_2} t_y^{-1} \cdot t_\gamma  t_{d_1}^{-1} . 
\end{align*}
By Lemma~\ref{lemma:10001} (2) and~\ref{lemma:10002}, $\psi_1$ maps $x_1$ to $\gamma_1$, $\omega_1$ maps $(y_1,\beta_1)$ to $(\delta_1,z_1)$, $\phi$ maps $(\delta,z)$ to $(y,d_2)$, and $\psi$ maps $d_1$ to $\gamma$. 
Lemma~\ref{lemma:67} gives the required formulas. 
\end{proof}

For a simple closed curve $c_1$ on $S_1$ appeared in the above, we set $c_i:=\Phi^{i-1}(c_1)$ which is a simple closed curve on $S_i$, and we write a lantern relation $L_i:={}_{\Phi^{i-1}}(L_1)$. 
From Lemma~\ref{example} and the primitive braid relation, we obtain 
\begin{align*}
\psi_i &= t_{\epsilon_i} t_{\gamma_i} t_{x_i} t_{\epsilon_i}, \\ 
\omega_i &= t_{\epsilon_i} t_{z_i} t_{\beta_i} t_{\epsilon_i} t_{\zeta_i} t_{\delta_i} t_{y_i} t_{\zeta_i}, \\
L_i &\e [t_{x_i},\psi_i]  [t_{y_i} t_{\beta_i}^{-1},\omega_i]  t_{\alpha_i}^{-1} 
\end{align*}
for $i=1,2,\ldots,12$. 
Moreover, we define a 2-chain relator $C_{2,1}$ to be
\begin{align*}
C_{2,1}&:=(t_{a_1}t_{b_1})^6 t_{s_{1,1}}^{-1}. 
\end{align*}

The following proposition is the key result to prove Theorem~\ref{thmA}. 
\begin{prop}\label{prop:10}
For $g\geq 39$, there are elements $\widetilde{\mathcal{A}}_1,\widetilde{\mathcal{B}}_1,\widetilde{\mathcal{C}}_1,\widetilde{\mathcal{D}}_1$ in $\CM_g^1$ such that 
\[L_1L_2\cdots L_{12} C_{2,1}L^{-1} \e [\widetilde{\mathcal{A}}_1,\widetilde{\mathcal{B}}_1][\widetilde{\mathcal{C}}_1,\widetilde{\mathcal{D}}_1],\]
\end{prop}

To prove Proposition~\ref{prop:10}, we prepare two lemmas (Lemma~\ref{twelvetwo} and~\ref{2chainlanterncommutator}). 
\begin{lem}\label{twelvetwo}
For $g\geq 39$, the following relation holds in $\CM_g^1$:
\begin{align*}
L_1L_2\cdots L_{12} \e t_{\alpha_{12}}^{-12} [X, \Psi] [YA, \Omega \Phi], 
\end{align*}
where $X := t_{x_1} t_{x_2} \cdots t_{x_{12}}$, $\Psi := \psi_1 \psi_2 \cdots \psi_{12}$, $Y := t_{y_1} t_{\beta_1}^{-1} t_{y_2} t_{\beta_2}^{-1} \cdots t_{y_{12}} t_{\beta_{12}}^{-1}$, $\Omega := \omega_1 \omega_2 \cdots \omega_{12}$, and $A :=t_{\alpha_1}^{-1} t_{\alpha_2}^{-2} \cdots t_{\alpha_{11}}^{-11}$. 
\end{lem}
\begin{proof}
Since then $S_i$ is disjoint from $S_{i^\prime}$ for $i\neq i^\prime$, any elements in $\CM(S_i)$ can commute with any elements in $\CM(S_{i^\prime})$ modulo $P$ from the commutative relations. 
Therefore, by $L_i\e [t_{x_i},\psi_i]  [t_{y_i} t_{\beta_i}^{-1},\omega_i]  t_{\alpha_i}^{-1} \in \CM(S_i)$ and Lemma~\ref{lemma:10}, we have 
\begin{align*}
L_1L_2\cdots L_{12} \e [X, \Psi] [Y, \Omega] t_{\alpha_1}^{-1}t_{\alpha_2}^{-1} \cdots t_{\alpha_{12}}^{-1}. 
\end{align*}
By Lemma~\ref{lemma:8} (1) and the definition of the curve $\alpha_i$, we obtain
\begin{align*}
L_1L_2\cdots L_{12} \e [X,\Psi] [Y,\Omega][A, \Phi] t_{\alpha_{12}}^{-12}. 
\end{align*}
Since $\alpha_i$ is disjoint from $\beta_i,\delta_i,\epsilon_i,\zeta_i,y_i,z_i$ for $i=1,2,\ldots,12$ and $S_i$ is disjoint from $S_{i^\prime}$ for $i\neq i^\prime$, $A$ can commute with $Y, \Omega$ modulo $P$ by the commutative relations. 
Here, $\omega_i,\Omega$ (resp. $t_{y_i}t_{\beta_i}^{-1}$, $Y$) and $\Phi$ satisfy the condition of Lemma~\ref{lemma:21} from the commutative and the primitive braid relations, so $\Phi$ can commute with $\Omega$ (resp. $Y$) modulo $P$. 
Lemma~\ref{lemma:10} and a cyclic permutation give the required formula. 
\end{proof}
The next lemma will be also used to prove Theorem~\ref{thmC}. 
\begin{lem}\label{2chainlanterncommutator}
There are elements $V',W'$ in $\CM(S)$ such that the following relation holds in $\CM(S)$:
\begin{align*}
C_{2,1}L^{-1} \e [V', W'] [t_\gamma,\psi] t_{a_1}^8 t_{b_1}^4. 
\end{align*}
\end{lem}
\begin{proof}
Let $C_3$ be the 3-chain relator in Definition~\ref{3-chain}. 
By the inclusion $\iota: \Sigma_1^2 \to \Sigma_1^1$ obtained by gluing a disk along $d'$, $\iota$ maps $(c,d)$ on $\Sigma_1^2$ to $(a,c)$ on $\Sigma_1^1$. 
Then, from the map $\iota_\ast:\CM_2^1 \to \CM_1^1$ induced by $\iota$, the trivial relation $t_{d'}=\mathrm{id}$ and the braid relation $t_at_bt_a=t_bt_at_b$ gives the 2-chain relator $C_2$ from $C_3$. 
From the equation (\ref{relation100}) in the case of $n=1$ and $\iota_\ast$, the equation \[C_2 \e t_a^8 t_b^4 [V,W] t_c^{-1}\] holds in $\CM_1^1$, where $V,W$ are some elements in $\CM_1^1$. 
Therefore, when we denote $S_1^1$ by the genus-1 subsurface bounded by $s_{1,1}$ as in Figure~\ref{L0L12}, Lemma~\ref{example} gives
\begin{align*}
C_{2,1}L^{-1} &\e t_{a_1}^8 t_{b_1}^4 [V_1,W_1] [t_\delta t_z^{-1},\phi] [t_\gamma,\psi], 
\end{align*}
where $V_1,W_1$ is in $\CM(S_1^1)$. 
Since $S_1^1$ is disjoint from $\delta, \zeta, \epsilon, y, z,d_2$, and $V_1,W_1$ are in $\CM(S_1^1)$, $V_1,W_1$ can commute with $t_\delta t_z^{-1},\phi$ modulo $P$ by the commutative relations. 
Lemma~\ref{lemma:10} and a cyclic permutation give the required formula. 
\end{proof}

We are now ready to prove Proposition~\ref{prop:10}. 
\begin{proof}[Proof of Proposition~\ref{prop:10}]
In the notation of Lemma~\ref{twelvetwo} and~\ref{2chainlanterncommutator}, since $S_i$ is disjoint from $S$ for any $i$, $YA, \Omega$ can commute with $V',W'$ modulo $P$ using the commutative relations. 
Here, for a simple closed curve $c$ on $S$, $\Omega\Phi(c)=c$ since $\Phi|_S = \mathrm{id}_S$ and $S_i$ is disjoint from $S$ for any $i$.
This gives $\Omega\Phi \cdot t_c \e t_c \cdot \Omega\Phi$, and therefore $\Omega\Phi$ can commute with $V',W'$ modulo $P$ since $V',W'$ are in $\CM(S)$. 
Hence, by Lemma~\ref{twelvetwo},~\ref{2chainlanterncommutator} and~\ref{lemma:10} and a cyclic permutation, we have
\begin{align*}
L_1L_2\cdots L_{12} \cdot C_{2,1}L^{-1} &\e  t_{\alpha_{12}}^{-12}[X,\Psi] [YA,\Omega\Phi] \cdot [V', W'] [t_\gamma,\psi] t_{a_1}^8 t_{b_1}^4  \\
&\e [t_\gamma,\psi] t_{a_1}^8 t_{b_1}^4 t_{\alpha_{12}}^{-12} [X,\Psi] [YAV',\Omega\Phi W']. 
\end{align*}
Note that $t_\gamma, \psi$ can commute with $t_{a_1},t_{b_1},t_{\alpha_{12}}$ modulo $P$ by the commutative relations since $\varepsilon, \gamma, d_1$ are disjoint from $a_1,b_1,\alpha_{12}$. 
Therefore, 
\[L_1L_2\cdots L_{12} \cdot C_{2,1}L^{-1} \e t_{a_1}^8 t_{b_1}^4 t_{\alpha_{12}}^{-12} [t_\gamma,\psi] [X,\Psi] [YAV',\Omega\Phi W']. \]
Since $S_i$ is disjoint from $S$, $t_\gamma$ and $\psi$ can commute with $X$ and $\Psi$ modulo $P$ from the commutative relations. 
Lemma~\ref{lemma:10} gives 
\begin{align}
L_1L_2\cdots L_{12} \cdot C_{2,1}L^{-1} &\e t_{a_1}^8 t_{b_1}^4 t_{\alpha_{12}}^{-12} [t_\gamma X, \psi \Psi] [YAV',\Omega\Phi W']. \label{equation123}
\end{align}
  \begin{figure}[hbt]
  \centering
       \includegraphics[scale=.90]{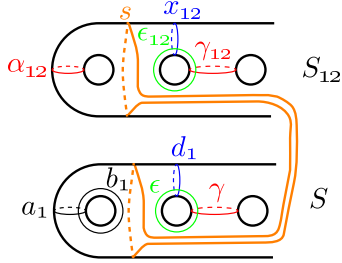}
       \caption{The separating curve $s$ on $\Sigma_g^1$.}
       \label{halftwist}
  \end{figure}

Let $S_2^1$ be the genus-$2$ subsurface of $\Sigma_g^1$ bounded by $s$ such that it contains simple closed curves $a_1,b_1,\alpha_{12}$ (see Figure~\ref{halftwist}) and disjoint from $d_1,\gamma,\epsilon$ and $x_i,\gamma_i,\epsilon_i$ for any $i=1,2,\ldots,12$. 
Then, we can consider the half twist $H_s$ along $s$ such that $H_s|_{\Sigma_g^1-S_2^1} = \mathrm{id}|_{\Sigma_g^1-S_2^1}$, $H_s(a_1)=\alpha_{12}$ and $H_s(\alpha_{12})=a_1$. 
Here we set 
\[H:=t_{a_1}t_{b_1}H_s.\]
Since $\alpha_{12}$ is disjoint from $a_1,b_1$, we see that $H|_{\Sigma_g^1-S_2^1} = \mathrm{id}|_{\Sigma_g^1-S_2^1}$, $H(\alpha_{12}) = b_1$ and $H(a_1) = \alpha_{12}$ from $t_{a_1}t_{b_1}(a_1)=b_1$ (see Definition~\ref{primitive}).
Therefore, Lemma~\ref{lemma:67} gives 
\begin{align*} 
t_{\alpha_{12}}^{-12} t_{a_1}^8 t_{b_1}^4 &\e t_{a_1}^8 t_{\alpha_{12}}^{-4} t_{b_1}^4 t_{\alpha_{12}}^{-8}  \\
&\e [t_{a_1}^8 t_{\alpha_{12}}^{-4}, H] .
\end{align*}
By this equation and the equation (\ref{equation123}), we obtain
\begin{align*}
L_1L_2\cdots L_{12} C_{2,1}L^{-1} &\e [t_{a_1}^8 t_{\alpha_{12}}^{-4}, H] [t_\gamma X, \psi\Psi] [YAV',\Omega\Phi W']. 
\end{align*}
Note that $a_1,b_1,\alpha_{12},s$ are disjoint from $d_1,\gamma,\epsilon$ and $\gamma_i,\epsilon_i,x_i$ for any $i=1,2,\ldots,12$. 
Hence, by $H|_{\Sigma_g^1-S_2^1} = \mathrm{id}|_{\Sigma_g^1-S_2^1}$, the definitions of $X,\Psi,\psi$ and the commutative relations, we see that $t_{a_1}^8 t_{\alpha_{12}}^{-4}$ and $H$ can commute with $t_\gamma X$ and $\psi\Psi$ modulo $P$. 
Lemma~\ref{lemma:10} gives
\begin{align*}
L_1L_2\cdots L_{12} C_{2,1}L^{-1} &\e [t_{a_1}^8 t_{\alpha_{12}}^{-4}t_\gamma X, \psi\Psi H] [YAV',\Omega\Phi W'], 
\end{align*}
and the proof is complete. 
\end{proof}

We show Theorem~\ref{thmA}. 
\begin{proof}[Proof of Theorem A]
Assume that $g\geq 39n$ and $n\geq 1$. 
Then, we can take $n$ disjoint subsurfaces $S_1^\prime,S_2^\prime,\ldots,S_n^\prime$ of $\Sigma_g^1$ of genus $39$ with one boundary component and find a diffeomorphism $\Phi^\prime$ on $\Sigma_g^1$ such that $\Phi^\prime(S_i^\prime)=S_{i+1}^\prime$. 
Let 
\begin{align*}
&R_1:= L_1L_2\cdots L_{12} C_{2,1}L^{-1}, \\ 
&R_{i+1}:={}_{\Phi^\prime}(R_i).
\end{align*}
Since $S_i^\prime$ is disjoint from $S_j^\prime$, by the commutative relations, Lemma~\ref{lemma:10} and Proposition~\ref{prop:10}, we have
\begin{align*}
R_1R_2\cdots R_n \e [\widetilde{\mathcal{A}},\widetilde{\mathcal{B}}][\widetilde{\mathcal{C}},\widetilde{\mathcal{D}}], 
\end{align*}
where $\widetilde{\mathcal{A}},\widetilde{\mathcal{B}},\widetilde{\mathcal{C}},\widetilde{\mathcal{D}}$ are some elements in $\CM_g^1$. 
In particular, we see that this relation also holds in $\CM_g$. 
This gives a  $\Sigma_g$-bundle $X \to \Sigma_2$ with a $0$-section for $g\geq 39n$. 
From the above argument, in the notation of Proposition~\ref{ehktprop0}, we have 
\begin{align*}
n(T) &= n^+(T) - n^-(T) = 0 - 0, \\
n(C_2) &= n^+(C_2) - n^-(C_2) = n - 0, \\
n(L) &= n^+(L) - n^-(L) = 12n - n. 
\end{align*}
This gives 
\[\sigma(X) = - 1 \cdot 0 - 7 \cdot n + 1 \cdot 11n = 4n\]
for $g\geq 39n$, and this finishes the proof. 
\end{proof}

\section{Surface bundles with fiber of odd genus}\label{odd}
This section shows Theorem~\ref{thmB} and~\ref{thmC}. 
To prove them, we prepare some results (Proposition~\ref{prop:11-1} and~\ref{prop:lanterncommutator} and Lemma~\ref{lemma:?}).

Let $\alpha_1,$ $\beta_1,$ $\gamma_1,$ $x_1,$ $y_1,$ $z_1,$ $x^\prime_1,$ $y^\prime_1,$ $z^\prime_1$ be the nonseparating curves on the genus-$2$ subsurface $S_2^2$ of $\Sigma_g$ bounded by $\delta_1,\delta_1^\prime$ as in Figure~\ref{Licurves}. 
We consider the following two lantern relators: 
\begin{align*}
L_1 &:= t_{\alpha_1}^{-1} t_{\delta_1}^{-1} t_{\gamma_1}^{-1} t_{\beta_1}^{-1} t_{x_1} t_{y_1} t_{z_1}, \\
L_1^\prime &:= t_{\beta_1}^{-1} t_{\gamma_1}^{-1} t_{\delta'_1}^{-1} t_{\alpha_1}^{-1} t_{x^\prime_1} t_{y^\prime_1} t_{z^\prime_1}. 
\end{align*}
\begin{prop}\label{prop:11-1}
For any integer $n$, there are elements $X_{1,1},Y_{1,1}, X_{2,1},Y_{2,1},\ldots,$ $X_{|n|+2,1}, Y_{|n|+2,1}$ in $\CM(S_2^2)$ such that the following holds in $\CM(S_2^2)$$:$
\begin{align*}
(L_1)^{2n}(L_1^\prime)^{2n} \e [X_{1,1},Y_{1,1}] [X_{2,1},Y_{2,1}] \cdots [X_{|n|+2,1}, Y_{|n|+2,1}] \cdot t_{\delta_1}^{-2n} t_{\delta_1^\prime}^{-2n}. 
\end{align*}
\end{prop}

  \begin{figure}[hbt]
  \centering
       \includegraphics[scale=.65]{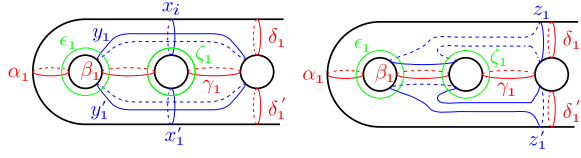}
       \caption{The curves $\alpha_1,\beta_1,\gamma_1,\delta_1,$ $\epsilon_1,\zeta_1,x_1,y_1,z_1,x^\prime_1,y^\prime_1,z^\prime_1$ on $S_2^2$.}
       \label{Licurves}
  \end{figure}

\begin{proof}
Note that $\alpha_1,\beta_1,\gamma_1,\delta_1$ (resp. $\alpha_1,\beta_1,\gamma_1,\delta_1^\prime$) are disjoint from each other and $z_1$ (resp. $x_1^\prime$). 
Therefore, by the lantern relations $t_{x_1} t_{y_1} t_{z_1} =t_{\beta_1} t_{\gamma_1} t_{\delta_1} t_{\alpha_1}$ and $t_{x_1^\prime} t_{y_1^\prime} t_{z_1^\prime} = t_{\alpha_1} t_{\gamma_1} t_{\beta_1} t_{\delta_1^\prime}$ and the commutative relations, we have 
\begin{align*}
t_{x_1} t_{y_1} t_{\alpha_1}^{-1} (z_1) &= t_{\beta_1} t_{\gamma_1} t_{\delta_1} t_{z_1}^{-1} (z_1) = z_1, \\
t_{y_1^\prime} t_{z_1^\prime} t_{\gamma_1}^{-1} (x_1^\prime) &= t_{x_1^\prime}^{-1} t_{\alpha_1} t_{\beta_1} t_{\delta_1^\prime} (x_1^\prime) = x_1^\prime. 
\end{align*}
Using these facts and the primitive braid relations, we obtain 
\begin{align*}
t_{x_1} t_{y_1} t_{\alpha_1}^{-1} \cdot t_{z_1} &\e t_{z_1} \cdot t_{x_1} t_{y_1} t_{\alpha_1}^{-1}, \\
t_{x_1^\prime} \cdot t_{y_1^\prime} t_{z_1^\prime} t_{\gamma_1}^{-1} &\e t_{y_1^\prime} t_{z_1^\prime} t_{\gamma_1}^{-1} \cdot t_{x_1^\prime}. 
\end{align*}
These two relations and the commutative relations give 
\begin{align*}
&(L_1)^{2n} \e (t_{y_1} t_{z_1} t_{\alpha_1}^{-1})^{2n} t_{x_1}^{2n} t_{\beta_1}^{-2n} t_{\gamma_1}^{-2n} t_{\delta_1}^{-2n} , \\
&(L_1^\prime)^{2n} \e (t_{x^\prime_1} t_{y^\prime_1} t_{\gamma_1}^{-1})^{2n} t_{z^\prime_1}^{2n} t_{\alpha_1}^{-2n} t_{\beta_1}^{-2n} t_{\delta_1^\prime}^{-2n} . 
\end{align*}
Since $x_1,y_1,z_1,x^\prime_1,y^\prime_1,z^\prime_1$ are disjoint from $\alpha_1,\beta_1,\gamma_1,\delta_1$, and $x_1,y_1,z_1$ are disjoint from $x^\prime_1,y^\prime_1,z^\prime_1$, by the commutative relations, we have 
\begin{align*}
(L_1)^{2n} (L_1^\prime)^{2n} &\e (t_{y_1} t_{z_1} t_{\gamma_1}^{-1} t_{x^\prime_1} t_{y^\prime_1} t_{\alpha_1}^{-1})^{2n} \cdot t_{x_1}^{2n} t_{z^\prime_1}^{2n} \cdot t_{\alpha_1}^{-2n} t_{\beta_1}^{-4n} t_{\gamma_1}^{-2n} \cdot t_{\delta_i}^{-2n} t_{\delta_i^\prime}^{-2n}. 
\end{align*}
Since $y_1,z_1$ are disjoint from $x_1^\prime,y_1^\prime$, and $\alpha_1,\gamma_1$ are disjoint from $y_1,z_1,x_1^\prime,y_1^\prime$, by the commutative and the primitive braid relations, we obtain 
\begin{align*}
(t_{y_1} t_{z_1} t_{\gamma_1}^{-1} t_{x^\prime_1} t_{y^\prime_1} t_{\alpha_1}^{-1})^2 &\e t_{y_1} t_{\alpha_1}^{-1} t_{x^\prime_1} t_{\gamma_1}^{-1} \cdot (t_{z_1}t_{y_1}t_{z_1}^{-1}) t_{\gamma_1}^{-1} (t_{y^\prime_1}t_{x^\prime_1}t_{y^\prime_1}^{-1}) t_{\alpha_1}^{-1} \cdot t_{z_1}^2 t_{y^\prime_1}^2 \\
&\e t_{y_1} t_{\alpha_1}^{-1} t_{x^\prime_1} t_{\gamma_1}^{-1} \cdot  t_{t_{y^\prime_1}(x^\prime_1)} t_{\gamma_1}^{-1} t_{t_{z_1}(y_1)} t_{\alpha_1}^{-1} \cdot t_{z_1}^2 t_{y^\prime_1}^2. 
\end{align*}
Here, let $f_1:=t_{z_1}t_{y^\prime_1} \cdot t_{\epsilon_1}t_{y_1}t_{\alpha_1}t_{\epsilon_1} \cdot t_{\zeta_1}t_{x_1^\prime}t_{\gamma_1}t_{\zeta_1}$ in $\CM(S_2^2)$. 
By the latter part of Lemma~\ref{lemma:10002}, $t_{\epsilon_1}t_{y_1}t_{\alpha_1}t_{\epsilon_1} \cdot t_{\zeta_1}t_{x_1^\prime}t_{\gamma_1}t_{\zeta_1}$ maps $(y_1,\alpha_1,x^\prime_1,\gamma_1)$ to $(\alpha_1,y_1,\gamma_1,x^\prime_1)$. 
From that $\alpha_1,\gamma_1,y_1$ are disjoint from $y'_1$, $t_{y_1'}$ maps $(\alpha_1,y_1,\gamma_1,x^\prime_1)$ to $(\alpha_1,y_1,\gamma_1,t_{y_1'}(x^\prime_1))$. 
Note that $y_1'$ and $x_1'$ are disjoint from $z_1$, so $t_{y^\prime_1}(x^\prime_1)$ is disjoint from $z_1$. 
From this, $t_{z_1}$ maps $(\alpha_1,y_1,\gamma_1,t_{y_1'}(x^\prime_1))$ to $(\alpha_1,t_{z_1}(y_1),\gamma_1,t_{y^\prime_1}(x^\prime_1))$ since $\alpha_1,\gamma_1,t_{y^\prime_1}(x^\prime_1)$ are disjoint from $z_1$. 
Therefore, we see that $f_1$ maps $(y_1,\alpha_1,x^\prime_1,\gamma_1)$ to $(\alpha_1,t_{z_1}(y_1),\gamma_1,t_{y^\prime_1}(x^\prime_1))$. 
From Lemma~\ref{lemma:67}, we obtain
\begin{align*}
t_{y_1} t_{\alpha_1}^{-1} t_{x^\prime_1} t_{\gamma_1}^{-1} \cdot  t_{t_{y^\prime_1}(x^\prime_1)} t_{\gamma_1}^{-1} t_{t_{z_1}(y_1)} t_{\alpha_1}^{-1} \e [t_{y_1} t_{\alpha_1}^{-1} t_{x^\prime_1} t_{\gamma_1}^{-1}, f_1]. 
\end{align*}
When we write $[X,Y]=[t_{y_1} t_{\alpha_1}^{-1} t_{x^\prime_1} t_{\gamma_1}^{-1}, f_1]$, we have  
\begin{align*}
(t_{y_1} t_{z_1} t_{\gamma_1}^{-1} t_{x^\prime_1} t_{y^\prime_1} t_{\alpha_1}^{-1})^2 \e [X,Y] t_{z_1}^2 t_{y^\prime_1}^2. 
\end{align*}
Since $z_1$ is disjoint from $y_1'$, the commutative relations and Lemma~\ref{lemma:31} (1) give 
\begin{align*}
([X,Y] \cdot t_{z_1}^2 t_{y^\prime_1}^2)^n &= \prod_{i=1}^n [X_i,Y_i] \cdot (t_{z_1}^2 t_{y^\prime_1}^2)^{n} \e \prod_{i=1}^n [X_i,Y_i] \cdot t_{z_1}^{2n} t_{y^\prime_1}^{2n},
\end{align*} 
where $[X_i,Y_i]={}_{(t_{z_1}^2 t_{y^\prime_1}^2)^{i-1}}([X,Y])$, which is a commutator since the conjugation of a commutator is also a commutator. 
Note that $\alpha_1,\beta_1,\gamma_1$ (resp. $z_1$) are disjoint from $x_1,z_1,y_1^\prime,z_1^\prime$ (resp. $y_1^\prime$). 
From the above arguments and the commutative relations give 
\begin{align*}
(L_1)^{2n} (L_1^\prime)^{2n} &\e \prod_{i=1}^n[X_i,Y_i] \cdot t_{z_1}^{2n} t_{y^\prime_1}^{2n} \cdot t_{x_1}^{2n} t_{z^\prime_1}^{2n} \cdot t_{\alpha_1}^{-2n} t_{\beta_1}^{-4n} t_{\gamma_1}^{-2n} \cdot t_{\delta_1}^{-2n} t_{\delta_1^\prime}^{-2n} \\
&\e \prod_{i=1}^n[X_i,Y_i] \cdot t_{z_1}^{2n} t_{\alpha_1}^{-2n} t_{y^\prime_1}^{2n} t_{\beta_1}^{-2n} \cdot t_{x_1}^{2n} t_{\beta_1}^{-2n} t_{z^\prime_1}^{2n} t_{\gamma_1}^{-2n} \cdot t_{\delta_1}^{-2n} t_{\delta_1^\prime}^{-2n}. 
\end{align*}
We set $f_2 = t_{\zeta_1} t_{\beta_1} t_{\epsilon_1} t_{y_1^\prime} t_{\alpha_1} t_{\epsilon_1} t_{z_1} t_{\zeta_1}$ and $f_3 = t_{\epsilon_1} t_{z^\prime_1} t_{\zeta_1} t_{\gamma_1} t_{x_1} t_{\zeta_1} t_{\beta_1} t_{\epsilon_1}$ in $\CM(S_2^2)$. 
By Lemma~\ref{lemma:10002}, $f_2$ and $f_3$ in $\CM(S_2^2)$ map $(z_1,\alpha_1)$ and $(x_1,\beta_1)$ to $(\beta_1,y^\prime_1)$ and $(\gamma_1,z^\prime_1)$, respectively. 
Therefore, by Lemma~\ref{lemma:67}, we have 
\begin{align*}
t_{z_1}^{2n} t_{\alpha_1}^{-2n} t_{y^\prime_1}^{2n} t_{\beta_1}^{-2n} &= [t_{z_1}^{2n} t_{\alpha_1}^{-2n}, f_2],\\ 
t_{x_1}^{2n} t_{\beta_1}^{-2n} t_{z^\prime_1}^{2n} t_{\gamma_1}^{-2n} &= [t_{x_1}^{2n} t_{\beta_1}^{-2n}, f_3], 
\end{align*}
and the proposition follows. 
\end{proof}

\begin{prop}\label{prop:lanterncommutator}
Suppose that $g$ is odd. Let $s_0$ be a nonseparating curve on $\Sigma_g$, and let $\mathcal{L}_i$ be a lantern relator in $\CM_g$. 
For any integer $n$, there are some elements $\mathcal{X}_1,\mathcal{Y}_1,\mathcal{X}_2,\mathcal{Y}_2,\ldots,\mathcal{X}_{|n|+2},\mathcal{Y}_{|n|+2}$ in $\CM_g$ such that 
\begin{align*}
\prod_{i=1}^{2(g-1)n} \mathcal{L}_i \e \prod_{j=1}^{|n|+2} [\mathcal{X}_j, \mathcal{Y}_j] \cdot t_{s_0}^{-2(g-1)n}. 
\end{align*}
\end{prop}
\begin{proof}[Proof of Proposition~\ref{prop:lanterncommutator}]
If $g=3$, Proposition~\ref{prop:lanterncommutator} immediately follows from $\delta_1=\delta_1^\prime$ and Proposition~\ref{prop:11-1}. 
\begin{figure}[hbt]
  \centering
       \includegraphics[scale=.75]{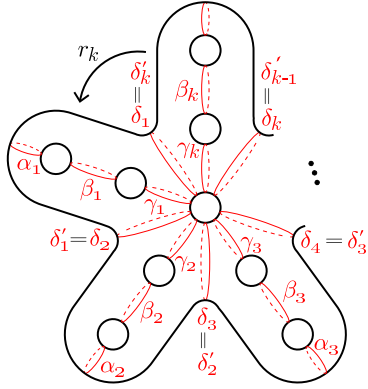}
       \caption{The rotation $r_k$ of $\Sigma_g$ for $g=2k+1$.}
       \label{rk}
  \end{figure}

If $g=2k+1$ and $k\geq 2$, then there is a rotation $r_k$ of $\Sigma_g$ by $2\pi/k$ as in Figure~\ref{rk}. 
In the notation of Proposition~\ref{prop:11-1}, we write 
\begin{align*}
&L_j := {}_{r_k^{j-1}}(L_1),& &L_j^\prime := {}_{r_k^{j-1}}(L_1^\prime),& \\ 
&\delta_j := r_k^{j-1}(\delta_1),& &\delta_j^\prime := r_k^{j-1}(\delta_1^\prime), \\
&X_{i,j} := {}_{r_k^{j-1}}(X_i),& &Y_{i,j} := {}_{r_k^{j-1}}(Y_i),& 
\end{align*}
for $j=1,2,\ldots,k$. 
Note that $\delta_{j+1} = \delta_j^\prime$, $\delta_{k+1}=\delta_1$, $X_{i,k+1} = X_{i,1}$ and $Y_{i,k+1} = Y_{i,1}$. 
For $j=1,2,\ldots,k$, Proposition~\ref{prop:11-1} and the primitive braid relations give 
\[ (L_j)^{2n} (L_j^\prime)^{2n} \e [X_{1,j},Y_{1,j}] [X_{2,j},Y_{2,j}] \cdots [X_{|n|+2,j}, Y_{|n|+2,j}] \cdot t_{\delta_j}^{-2n} t_{\delta_j^\prime}^{-2n}. \]

Recall that $S_2^2$ is the genus-$2$ subsurface of $\Sigma_g$ bounded by $\delta_1,\delta_1^\prime$. 
Any simple closed curves on $\mathrm{Int}(r_k^{j-1}(S_2^2))$ are disjoint from any simple closed curves on $\mathrm{Int}(r_k^{j^\prime-1}(S_2^2))$ if $j\neq j^\prime$, and $\delta_j,\delta_j^\prime$ are boundary curves of $r_k^{j-1}(S_2^2)$. 
Hence, for any elements $e_j$ in $\CM(r_k^{j-1}(S_2^2))$ and any element $f_{j^\prime}$ in $\CM(r_k^{j^\prime-1}(S_2^2))$, we have $e_j f_{j^\prime} \e f_{j^\prime} e_j$ for $j\neq j^\prime$ by the commutative relations and the property of boundary curves. 
By $\delta_{j+1} = \delta_j^\prime$, $\delta_{k+1}=\delta_1$, the commutative relations and Lemma~\ref{lemma:10} we have 
\[\prod_{j=1}^k (L_j)^{2n}(L_j^\prime)^{2n} \e [\mathcal{X}_1,\mathcal{Y}_1] [\mathcal{X}_2,\mathcal{Y}_2] \cdots [\mathcal{X}_{|n|+2},\mathcal{Y}_{|n|+2}] \ t_{\delta_1}^{-4n} t_{\delta_2}^{-4n} \cdots t_{\delta_k}^{-4n}, \] 
where $\mathcal{X}_i = X_{i,1}X_{i,2} \cdots X_{i,k}$ and $\mathcal{Y}_i = Y_{i,1}Y_{i,2} \cdots Y_{i,k}$, and Lemma~\ref{lemma:8} (1) gives 
\begin{align*}
t_{\delta_1}^{-4n} t_{\delta_2}^{-4n} \cdots t_{\delta_k}^{-4n} \e [t_{\delta_1}^{-4n} t_{\delta_2}^{-8n} \cdots t_{\delta_{k-1}}^{-4(k-1)n},r_k] \cdot t_{\delta_k}^{-4kn}. 
\end{align*}
Since from their definition, $\mathcal{X}_{|n|+2}, X_{|n|+2,j}$ (resp. $\mathcal{Y}_{|n|+2}, Y_{|n|+2,j}$) and $r_k$ satisfy the condition of Lemma~\ref{lemma:21}, by the primitive braid relations, we obtain $\mathcal{X}_{|n|+2} r_k \e r_k \mathcal{X}_{|n|+2}$ (resp. $\mathcal{Y}_{|n|+2} r_k \e r_k \mathcal{Y}_{|n|+2}$). 
Moreover, since $\delta_j$ is a boundary curve of $r_k^{j-1}(S_2^2)$ and disjoint from $r_k^{j'-1}(S_2^2)$ if $j\neq j'$, $\mathcal{X}_{|n|+2}$ and $\mathcal{Y}_{|n|+2}$ can commute with $t_{\delta_j}$ modulo $P$ for any $j$ by the commutative relations and the property of boundary curves. 
From the above argument, Lemma~\ref{lemma:10} gives
\begin{align*}
[\mathcal{X}_{|n|+2},\mathcal{Y}_{|n|+2}] [t_{\delta_1}^{-4n} t_{\delta_2}^{-8n} \cdots t_{\delta_{k-1}}^{-4(k-1)n}, r_k] = [\mathcal{X}_{|n|+2} t_{\delta_1}^{-4n} t_{\delta_2}^{-8n} \cdots t_{\delta_{k-1}}^{-4(k-1)n}, \mathcal{Y}_{|n|+2} r_k], 
\end{align*}
and we obtain the desired conclusion. 
\end{proof}

\begin{lem}[\cite{ko2}]\label{lemma:?}
Let us consider the lantern relator $L:=t_a^{-2}t_d^{-1}t_{d'}^{-1} t_ct_{s_1}t_z$, the 2-chain relator $C_2:=t_{s_1}^{-1} (t_at_b)^6$ and the 3-chain relator $C_3:=t_{d}^{-1}t_{d'}^{-1}(t_at_bt_c)^4$, where the curves are as in Figure~\ref{abcdd'}. 
\[C_3 \e L \cdot C_2.\] 
\end{lem}
\begin{proof}
Since $a,d,d'$ are disjoint from $c,z$ and each other, $t_a,t_d,t_{d'}$ can commute with $t_c,t_z$ modulo $P$ by the commutative relations. 
Combining this with a cyclic permutation give $L \e t_zt_c t_d^{-1}t_{d'}^{-1} t_{a}^{-2} t_{s_1}$. 
Here, by the braid relation, we have $t_at_bt_at_bt_at_b \e t_at_at_bt_at_at_b$. 
Therefore, using a cyclic permutation we have
\begin{align*}
L \cdot C_2 \e t_d^{-1} t_{d'}^{-1} \cdot t_bt_at_at_b \cdot t_at_at_bt_at_at_b \cdot t_zt_c. 
\end{align*}
By drawing corresponding curves and applying the corresponding Dehn twist, we find that $t_bt_at_at_b(z)=c$. 
This gives $t_bt_at_at_b \cdot t_z \e t_c \cdot t_bt_at_at_b$ by the primitive braid relation. 
Using this equation, we have 
\begin{align*}
L \cdot C_2 \e t_d^{-1} t_{d'}^{-1} \cdot \underline{t_bt_at_at_b \cdot t_at_a \cdot t_c \cdot t_bt_a} t_at_b \cdot t_c. 
\end{align*}
We focus on the underlined part. 
By Lemma~\ref{lemma:10001}, we have $t_bt_at_at_b(a)=a$, $t_at_bt_c(b)=c$ and $t_at_bt_c(a)=b$. 
This gives $t_bt_at_at_b \cdot t_a \e t_a \cdot t_bt_at_at_b$, $t_at_bt_c \cdot t_b \e t_c \cdot t_at_bt_c$ and $t_at_bt_c \cdot t_a \e t_b \cdot t_at_bt_c$. 
Applying them on the underlined parts, we obtain
\begin{align*}
\underline{t_bt_a t_a t_b t_a t_a} t_c t_b t_a \e t_a t_a t_b t_a \underline{t_a t_b t_c t_b t_a} \\
\e t_a t_a t_b t_a t_c t_b t_a t_b t_c. 
\end{align*}
By the braid relation and $t_at_bt_c \cdot t_b \e t_c \cdot t_at_bt_c$ on the underlined parts, we get
\begin{align*}
t_a \underline{t_a t_b t_a} t_c t_b t_a t_b t_c \e t_a t_b \underline{t_a t_b t_c t_b} t_a t_b t_c \\
\e t_a t_b t_c t_a t_b t_c t_a t_b t_c. 
\end{align*}
This finishes the proof. 
\end{proof}

We now prove Theorem~\ref{thmB}. 
\begin{proof}[Proof of Theorem B]
We may assume that two simple closed curves in Proposition~\ref{prop:lanterncommutator} and Theorem~\ref{thm:1000} are same from the primitive braid relation since there is an element $f$ in $\CM_g$ such that $f(c)=c^\prime$ for any two nonseparating curve $c,c^\prime$.

After inserting $5n$ into $n$ in Proposition~\ref{prop:lanterncommutator} and applying a cyclic permutation, by Theorem~\ref{thm:1000} we get
\begin{align*}
\prod_{i=1}^{10(g-1)n} \mathcal{L}_i \cdot \prod_{j=1}^{g-1} C_{3,j}^n \e \prod_{j=1}^{5|n|+3} [\mathcal{X}_j, \mathcal{Y}_j] \cdot \prod_{j=1}^{|n|+3} [\mathcal{V}_j, \mathcal{W}_j], 
\end{align*}
This gives an $\Sigma_g$-bundle $X\to \Sigma_{6|n|+6}$ for odd $g$ (This construction is called the ``subtraction of Lefschetz fibration" introduced in \cite{EKKOS}). 
By Lemma~\ref{lemma:?}, we see that
\begin{align*}
n(T) &= n^+(T) - n^-(T) = 0 - 0, \\
n(C_2) &= n^+(C_2) - n^-(C_2) = (g-1)n - 0, \\
n(L) &= n^+(L) - n^-(L) = 11(g-1)n - 0 
\end{align*}
in the notation of Proposition~\ref{ehktprop0}. Therefore, we have 
\begin{align*}
\sigma(X) = - 1\cdot 0 - 7 \cdot (g-1)n + 1\cdot 11(g-1)n = 4(g-1)n.
\end{align*}
This completes the proof. 
\end{proof}
\begin{rem}
We don't know the surface bundles constructed in Theorem~\ref{thmB} admits a section or not. 
\end{rem}

In the rest of this section, we prove Theorem~\ref{thmC}. 
\begin{proof}[Proof of Theorem C]
Let us consider the two (sub)surfaces of genus $3$ with one boundary component as in Figure \ref{Licurves} and the left side of Figure~\ref{L0L12}. 
Since $a_1,b_1$ are disjoint from $\gamma, \epsilon, d_1$, $t_{a_1},t_{b_1}$ can commute with $t_\gamma,\psi(=t_\epsilon t_\gamma t_{d_1} t_\epsilon)$ modulo $P$ by the commutative relations. 
Therefore, a cyclic permutation and Lemma~\ref{2chainlanterncommutator} give 
\begin{align*}
C_{2,1}L^{-1} &\e t_{a_1}^8 t_{b_1}^4 [t_\gamma,\psi] [Vt_\delta t_z^{-1}, W\phi]. 
\end{align*}
\begin{figure}[hbt]
  \centering
       \includegraphics[scale=.90]{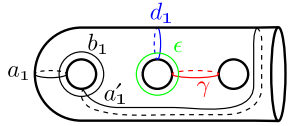}
       \caption{The curves $a_1,a_1',b_1,d_1,\gamma,\epsilon$ on $S$.}
       \label{curvea'}
  \end{figure}
Here, there is an element $f$ in $\CM_3^1$ such that $f(\delta_1)=a_1$ and $f(\delta_1')=a_1'$, where $a_1'$ is the simple closed curve as in Figure~\ref{curvea'}. 
By Proposition~\ref{prop:11-1} and the primitive braid and the commutative relations, we have 
\begin{align*}
{}_{f}\left((L_1)^6 (L_1^\prime)^6\right) C_{2,1}L^{-1} \e [X_1,Y_1] [X_2,Y_2] \cdots [X_5,Y_5] t_{a_1'}^{-6} t_{a_1}^2 t_{b_1}^4 [t_\gamma,\psi] [Vt_\delta t_z^{-1}, W\phi].
\end{align*}
in $\CM_3^1$, where $X_i={}_{f}(X_{i,1})$, $Y_i={}_{f}(Y_{i,1})$. 
Let $f_5:=t_{b_1}^2 \cdot t_{a_1'}t_{b_1}t_{a_1}$ in $\CM_3^1$. 
Since then $t_{a_1'}t_{b_1}t_{a_1}$ maps $(a_1',b_1)$ to $(b_1,a_1)$ by Lemma~\ref{lemma:10001} (2), we see that $f_5$ maps $(a_1',b_1)$ to $(b_1,{}_{t_{b_1}^2}(a_1))$. 
Therefore, the primitive braid relation and Lemma~\ref{lemma:67} give 
\begin{align*}
t_{a_1'}^{-6} t_{a_1}^2 t_{b_1}^4 &= t_{a_1'}^{-6} t_{b_1}^{-2} (t_{b_1}^2 t_{a_1}^2 t_{b_1}^{-2}) t_{b_1}^6 \\
&\e t_{a_1'}^{-6} t_{b_1}^{-2} t_{t_{b_1}^2(a_1)}^2 t_{b_1}^6 \\
&\e [t_{a_1'}^{-6} t_{b_1}^{-2}, f_5]. 
\end{align*}
Note that $t_{a_1'},t_{b_1},f_5$ can commute with $t_\gamma,\psi$ by the commutative relations since $a_1,a_1',b_1$ are disjoint from $\gamma, \epsilon, d_1$. 
By the above argument, Lemma~\ref{lemma:10} gives
\begin{align*}
{}_{f}\left((L_1)^6 (L_1^\prime)^6\right) C_{2,1}L^{-1} \e [X_1,Y_1] [X_2,Y_2] \cdots [X_5,Y_5] [t_{a_1'}^{-6} t_{b_1}^{-2}t_\gamma,f_5\psi] [Vt_\delta t_z^{-1}, W\phi].
\end{align*}
in $\CM_3^1$. 
In particular, this equation holds in $\CM_3$, so we get an $\Sigma_3$-bundle over $\Sigma_7$ with a $0$-section. 
Therefore, we have
\begin{align*}
n(T) &= n^+(T) - n^-(T) = 0 - 0, \\
n(C_2) &= n^+(C_2) - n^-(C_2) = 1 - 0, \\
n(L) &= n^+(L) - n^-(L) = 12 - 1 
\end{align*}
in the notation of Proposition~\ref{ehktprop0}, and 
\begin{align*}
\sigma(X) = -0 - 7 \cdot 1 + 1\cdot 11 = 4. 
\end{align*}
The proof is complete. 
\end{proof}

\section{Proofs of Theorem E (1) and (2)}
Since we don't use the results proved from here to compute signatures of surface bundles, replacing ``$\e$" by ``$=$" and ignoring the numbers of the relators $L,T,C_2$ pose no problem. 
From now on, we do not write $\e$ and relators explicitly.

We use the next result to prove Theorem~\ref{thmE} (1). 
\begin{thm}[Bavard \cite{ba}]\label{bavard}
Let $h_1, g_1, h_2, g_2, \ldots, h_k, g_k$ be elements in a group $G$. 
Then, for any integer $n$, $([h_1, g_1][h_2, g_2] \cdots [h_k, g_k])^n$ is written as a product of $|n|(k-1)+\left[\frac{|n|}{2}\right]+1$ commutators. 
\end{thm}
\begin{proof}[Proof of Thoerem E (1)]
In the notation of Proposition~\ref{prop:2}, if $g=1$, then $d$ and $d'$ are trivial. 
Therefore, we see that $t_{b}^{12n}$ can be written as a product of $|n|+1$ commutators in $\CM_1$. 
This gives $\cl_{\CM_1}(t_{b}^{12n})\leq |n|+1$ for any integer $n$.

To obtain a contradiction with $\scl_{\CM_1}(t_{s_0})=1/12$ (see Section~\ref{scl}), suppose that for some integer $k\geq 1$, $t_{b}^{12k}$ can be written as a product of $k$ commutators. 
Then, Theorem~\ref{bavard} gives 
\begin{align*}
\cl_{\CM_1}(t_{b}^{12kn}) \leq n(k-1)+ \left[\frac{n}{2}\right]+1.
\end{align*}
for any positive integer $n$. Therefore, we have
\begin{align*}
\scl_{\CM_1}(t_{b}^{12k}) \leq (k-1)+\frac{1}{2} = k-\frac{1}{2}.
\end{align*}
Since $\scl_{\CM_1}(t_{b_1}) = \scl_{\CM_1}(t_{b}^{12k})/12k$ (see Section~\ref{scl}), we obtain 
\begin{align*}
\scl_{\CM_1}(t_{b}) \leq \dfrac{1}{12} - \dfrac{1}{24k} < \dfrac{1}{12}. 
\end{align*}
This contradicts our assumption, which proves the theorem. 
\end{proof}

Next, we give a proof of Theorem~\ref{thmE} (2). 
\begin{proof}[Proof of Theorem E (2)]
\begin{figure}[hbt]
  \centering
       \includegraphics[scale=.90]{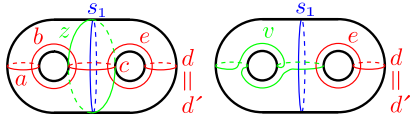}
       \caption{$\Sigma_2$ and the curves $a,b,c,d,e,v$ on $\Sigma_1^2$.}
       \label{genus2surface}
  \end{figure}
We consider the relation (\ref{relation10000}) and embed $\Sigma_1^2$ into $\Sigma_2$ so that $d=d'$ (see Figure~\ref{genus2surface}). 
Lemma~\ref{lemma:31} (2) and the primitive braid relation ${}_{t_{b}^{-2i+2}} (t_{v}) = {}_{t_{b}^{-2i+1}} (t_{t_{b}(v)})$ give  
\begin{align*}
(t_{b} t_{v})^{2n} &\e t_{b}^{2n} \prod_{i=n}^1 {}_{t_{b}^{-2i+1}} (t_{v}t_{t_{b}(v)}). 
\end{align*}
Here, the primitive braid relations give 
\begin{align*}
t_{a}^{4n} t_{c}^{4n} t_{b}^{2n} &= t_{b}^{10n} \cdot t_{b}^{-4n} (t_{b}^{-6n} t_{a}^{4n} t_{b}^{6n}) t_{b}^{-4n} (t_{b}^{-2n} t_{c}^{4n} t_{b}^{2n}) \\
&\e t_{b}^{10n} \cdot t_{b}^{-4n} t_{t_{b}^{-6n}(a)}^{4n} t_{b}^{-4n} t_{t_{b}^{-2n} (c)}^{4n}. 
\end{align*}
By combining the above two relations with the relation (\ref{relation10000}) and using the commutative relations, we obtain 
\begin{align*}
C_3^n \e   t_{b}^{10n} \cdot t_{b}^{-4n} t_{t_{b}^{-6n}(a)}^{4n} t_{b}^{-4n} t_{t_{b}^{-2n} (c)}^{4n} \cdot \prod_{i=n}^1 {}_{t_{b}^{-2i+1}} (t_{v} t_{d}^{-1} t_{t_{b}(v)} t_{d}^{-1}). 
\end{align*}
Since $t_at_bt_c$ maps $(a,b)$ to $(b,c)$ by Lemma~\ref{lemma:10001} (2), we find that $t_b^{-2n}t_at_bt_ct_b^{6n}$, denoted $f_3$, maps $(t_{b}^{-6n}(a),b)$ to $(b,t_{b}^{-2n}(c))$. 
Let $e$ be a nonseparating curve as in Figure~\ref{genus2surface}. 
Since $t_et_dt_vt_e$ maps $(v,d)$ to $(d,v)$ by Lemma~\ref{lemma:10001} (1), $t_b t_et_dt_vt_e$, denoted $f_4$, maps $(v,d)$ to $(d,t_b(v))$ by $i(b,d)=0$. 
By Lemma~\ref{lemma:67} we see that 
\begin{align*}
t_{b}^{-4n} t_{t_{b}^{-6n}(a)}^{4n} t_{b}^{-4n} t_{t_{b}^{-2n} (c)}^{4n} &= [t_{b}^{-4n} t_{t_{b}^{-6n}(a)}^{4n},f_3], \\
t_{v} t_{d}^{-1} t_{t_{b}(v)} t_{d}^{-1} &= [t_{v} t_{d}^{-1},f_4]. 
\end{align*}
Since the conjugation of a commutator is also a commutator, Theorem~E~(2) follows. 
\end{proof}

\section{Scl of the Dehn twist along a separating curve}\label{Scl of the Dehn twist along a separating curve}

\subsection{A separating curve of type $1$}
We show Theorem~\ref{thmD} (2) and E (3).

We consider the subsurface $S_1^2$ in the proof of Theorem~\ref{thm:1000} and the curves $a_1,b_1,c_1,s_{1,1},z_1,d_1,d_{g-1}$ as in Figure~\ref{Li}. 
The separating curve $s_{1,1}$ is  of type $1$. 
\begin{prop}\label{separatinggenus1}
For any integer $n$, there are elements $V'_1,W'_1, V'_2,W'_2,\ldots,V'_{\left[\frac{|n|}{2}\right]+1},$ $W'_{\left[\frac{|n|}{2}\right]+1}$ in $\CM(S_1^2)$ such that the following holds in $\CM(S_1^2)$:
\[t_{s_{1,1}}^n = [V'_1,W'_1][V'_2,W'_2] \cdots [V'_{\left[\frac{|n|}{2}\right]+1}, W'_{\left[\frac{|n|}{2}\right]+1}] t_{d_{g-1}}^n t_{d_1}^n.\] 
\end{prop}
\begin{proof}
From the lantern relation $t_{c_1} t_{s_{1,1}} t_{z_1} = t_{d_1} t_{d_{g-1}} t_{a_1}^2$, we get $t_{s_{1,1}} = t_{c_1}^{-1} t_{d_1} t_{d_{g-1}} t_{a_1}^2 t_{z_1}^{-1}$. 
Since $a_1,d_1,d_{g-1}$ are disjoint from each other and $c_1,z_1$, using the commutative relation and Lemma~\ref{lemma:31} (1), we have 
\begin{align*}
t_{s_{1,1}}^n &= (t_{c_1}^{-1}t_{z_1}^{-1})^n t_{a_1}^n t_{a_1}^n t_{d_{g-1}}^n t_{d_1}^n \\
&= {}_{t_{c_1}^{-1}}(t_{z_1}^{-1}) {}_{t_{c_1}^{-2}}(t_{z_1}^{-1}) \cdots {}_{t_{c_1}^{-n}}(t_{z_1}^{-1}) \cdot t_{c_1}^{-n} t_{a_1}^n t_{a_1}^n t_{d_{g-1}}^n t_{d_1}^n. 
\end{align*}
From the commutative relations and the primitive braid relations ${}_{t_{c_1}^{-2i+1}}(t_{z_1}^{-1}) = {}_{t_{c_1}^{-2i}} (t_{t_{c_1}(z_1)}^{-1})$ and ${}_{t_{c_1}^{-2m-1}}(t_{z_1}^{-1}) = t_{t_{c_1}^{-2m-1}(z_1)}^{-1}$, we have 
\begin{align*}
t_{s_{1,1}}^{2m} = \prod_{i=1}^{m} {}_{t_{c_1}^{-2i}} (t_{t_{c_1}(z_1)}^{-1} t_{a_1} t_{z_1}^{-1} t_{a_1}) \cdot t_{c_1}^{-2m} t_{a_1}^{2m} \cdot t_{d_{g-1}}^{2m} t_{d_1}^{2m}
\end{align*}
and 
\begin{align*}
t_{s_{1,1}}^{2m+1} = \prod_{i=1}^{m} {}_{t_{c_1}^{-2i}} (t_{t_{c_1}(z_1)}^{-1} t_{a_1} t_{z_1}^{-1} t_{a_1})  \cdot t_{t_{c_1}^{-2m-1}(z_1)}^{-1} t_{a_1}^{2m+1} t_{c_1}^{-2m-1} t_{a_1} \cdot t_{d_{g-1}}^{2m+1} t_{d_1}^{2m+1}. 
\end{align*}
Since $t_{b_1}t_{a_1}t_{z_1}t_{b_1}$ maps $(z_1,a_1)$ to $(a_1,z_1)$ by Lemma~\ref{lemma:10001} (1), $t_{b_1}t_{a_1}t_{z_1}t_{b_1}t_{c_1}^i$ maps $(t_{c_1}^{-i}(z_1),a_1)$ to $(a_1,z_1)$ by $i(a_1,c_1)=0$. 
From the proof of Lemma~\ref{lemma:?}, $t_{b_1}t_{a_1}t_{a_1}t_{b_1}$ maps $(z_1,a_1)$ to $(c_1,a_1)$. 
Therefore, when we set $\phi_1:=t_{b_1}t_{a_1}t_{z_1}t_{b_1}t_{c_1}^{-1}$ and $\phi_2:=t_{b_1}t_{a_1}t_{a_1}t_{b_1} \cdot t_{b_1}t_{a_1}t_{z_1}t_{b_1} t_{c_1}^{2m+1}$, $\phi_1$ and $\phi_2$ maps $(t_{c_1}(z_1),a_1)$ and $(t_{c_1}^{-2m-1}(z_1),a_1)$ to $(a_1,z_1)$ and $(a_1,c_1)$, respectively. 
Moreover, $t_{b_1}t_{c_1}t_{a_1}t_{b_1}$, denoted $\phi_3$, maps $a_1$ maps $c_1$ by Lemma~\ref{lemma:10001} (1). 
Lemma~\ref{lemma:67} gives 
\begin{align*}
t_{t_{c_1}(z_1)}^{-1} t_{a_1} t_{z_1}^{-1} t_{a_1} &= [t_{t_{c_1}(z_1)}^{-1} t_{a_1}, \phi_1], \\
t_{t_{c_1}^{-2m-1}(z_1)}^{-1} t_{a_1}^{2m+1} t_{c_1}^{-2m-1} t_{a_1} &= [t_{t_{c_1}^{-2m-1}(z_1)}^{-1} t_{a_1}^{2m+1}, \phi_2], \\
t_{c_1}^{-2m} t_{a_1}^{2m} &= [t_{c_1}^{-2m}, \phi_3]. 
\end{align*}
Since the conjugation of a commutator is also a commutator, the proof is complete. 
\end{proof}

\begin{proof}[Proof of Theorem D (2) and E (3)]
If $g=2$, then by $d_1=d_{g-1}$, we have 
\[t_{s_{1,1}}^{5n} = [V'_1,W'_1][V'_2,W'_2] \cdots [V'_{\left[\frac{|5n|}{2}\right]+1}, W'_{\left[\frac{|5n|}{2}\right]+1}] \cdot t_{d_1}^{10n}.\]
By Theorem~\ref{thmE} (2), Theorem~\ref{thmE} (3) is proved.

Note that $s_{1,1}$ is a separating curve of genus-$1$. 
In the notation of proofs of Theorem~\ref{thm:1000} and Proposition~\ref{separatinggenus1}, we write
\begin{align*}
&s_{1,j} := r^{j-1}(s_{1,1}),& &d_j := r^{j-1}(d_1),& \\
&V'_{i,j} := {}_{r^{j-1}}(V'_i),& &W'_{i,j} := {}_{r^{j-1}}(W'_i)& 
\end{align*}
for $j=1,2,\ldots,g-1$. Note that $d_g=d_1$. 
Then, for $j=1,2,\ldots,g-1$, Proposition~\ref{separatinggenus1} and the primitive braid relations give 
\[t_{s_{1,j}}^{5n} = [V'_{1,j},W'_{1,j}][V'_{2,j},W'_{2,j}] \cdots [V'_{\left[\frac{|5n|}{2}\right]+1,j}, W'_{\left[\frac{|5n|}{2}\right]+1,j}] t_{d_j}^{5n} t_{d_{j+1}}^{5n}.\]
Here, any simple closed curves on $\mathrm{Int}(r^{j-1}(S_1^2))$ are disjoint from any simple closed curves on $\mathrm{Int}(r^{j^\prime-1}(S_1^2))$ if $j\neq j^\prime$, and $d_j,d_{j+1}$ are boundary curves of $r^{j-1}(S_1^2)$. 
Hence, for any elements $e_j$ in $\CM(r^{j-1}(S_1^2))$ and any element $f_{j^\prime}$ in $\CM(r^{j^\prime-1}(S_1^2))$, we have $e_j f_{j^\prime} = f_{j^\prime} e_j$ by the commutative relations and the property of boundary curves if $j\neq j^\prime$. 
When we set $\mathcal{V}'_i = V'_{i,1} V'_{i,2} \cdots V'_{i,g-1}$ and $\mathcal{W}'_i = W'_{i,1} W'_{i,2} \cdots W'_{i,g-1}$, from Lemma~\ref{lemma:10} and $d_g=d_1$, we have 
\begin{align*}
t_{s_{1,1}}^{5n} t_{s_{1,2}}^{5n} \cdots t_{s_{1,g-1}}^{5n} =  [\mathcal{V}'_1,\mathcal{W}'_1] [\mathcal{V}'_2,\mathcal{W}'_2] \cdots [\mathcal{V}'_{\left[\frac{|5n|}{2}\right]+1}, \mathcal{W}'_{\left[\frac{|5n|}{2}\right]+1}]\cdot t_{d_1}^{10n} t_{d_2}^{10n} \cdots t_{d_{g-1}}^{10n}. 
\end{align*}
Moreover, Lemma \ref{lemma:8} gives 
\begin{align*}
t_{s_{1,g-1}}^{5(g-1)n} [T_{s_{1,1}}, r] =  [\mathcal{V}'_1,\mathcal{W}'_1] [\mathcal{V}'_2,\mathcal{W}'_2] \cdots [\mathcal{V}'_{\left[\frac{|5n|}{2}\right]+1}, \mathcal{W}'_{\left[\frac{|5n|}{2}\right]+1}] \cdot t_{d_{g-1}}^{10(g-1)n} [T_d,r], 
\end{align*}
where $T_{s_{1,1}} = t_{s_{1,1}}^{5n} t_{s_{1,2}}^{10n} \cdots t_{s_{1,g-2}}^{5(g-2)n}$ and $T_d = t_{d_1}^{10n} t_{d_2}^{20n} \cdots t_{d_{g-2}}^{10(g-2)n}$. 
We obtain Theorem~\ref{thmD} (2) by $[T_d,r][T_{s_{1,1}}, r]^{-1} = [T_d,r][r, T_{s_{1,1}}]$, Lemma~\ref{XYYZ} and Theorem~\ref{thm:1000}. 
\end{proof}

\subsection{A separating curve of type $h$}
We give the proof of Theorem~\ref{thmD} (3). 

Let $a,b,c,d,e,x,y,z$ be the nonseparating curves on the genus-$h$ subsurface $S_h^1$ of $\Sigma_g$ bounded by the separating curve $s_h$ of type $s_h$ as in Figure~\ref{genus-h}. 
  \begin{figure}[hbt]
  \centering
       \includegraphics[scale=.90]{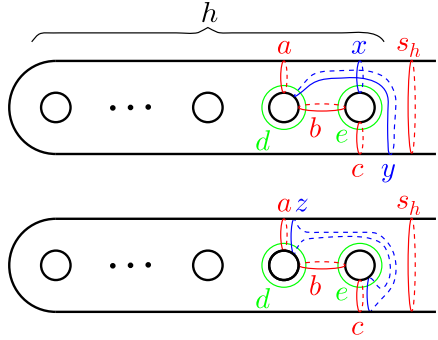}
       \caption{The curves $s_h,a,b,c,d,e,x,y,z$.}
       \label{genus-h}
  \end{figure}
\begin{prop}[\cite{bkm}]\label{separatinggenush}
For any integer $n$, there are elements $X_1,Y_1,X_2,Y_2,\ldots,$ $X_{\left[\frac{|n|+3}{2}\right]}, Y_{\left[\frac{|n|+3}{2}\right]}$ in $\CM(S_h^1)$ such that the following holds in $\CM(S_h^1)$:
\[t_{s_h}^n = [X_1,Y_1][X_2,Y_2] \cdots [X_{\left[\frac{|n|+3}{2}\right]}, Y_{\left[\frac{|n|+3}{2}\right]}]. \]
\end{prop}
\begin{proof}
By the lantern relation $t_xt_yt_z=t_bt_ct_{s_h}t_a$, we have $t_xt_y=t_bt_ct_{s_h}t_at_z^{-1}$. 
Since $a,b,c,d,s_h,z$ are disjoint from each other, by the commutative relations and Lemma~\ref{lemma:31} (1), we have
\begin{align*}
t_{s_h}^n &= (t_xt_y)^nt_z^nt_a^{-n}t_b^{-n}t_c^{-n} \\
&= {}_{t_x}(t_y) {}_{t_x^2}(t_y) \cdots {}_{t_x^n}(t_y) t_x^n t_z^n t_a^{-n}t_b^{-n}t_c^{-n}.
\end{align*}
From the primitive braid relations ${}_{t_x^{2i}}(t_{y}) = {}_{t_x^{2i-1}}(t_{t_x(y)})$ and ${}_{t_x^{2m+1}}(t_y) = t_{t_x^{2m+1}(y)}$ and the commutative relations, we obtain 
\begin{align*}
t_{s_h}^{2m} = \prod_{i=1}^m {}_{t_x^{2i-1}}(t_y t_a^{-1} t_{t_x(y)} t_a^{-1}) \cdot t_x^{2m} t_b^{-2m} t_z^{2m} t_c^{-2m}, 
\end{align*}
and 
\begin{align*}
t_{s_h}^{2m+1} = \prod_{i=1}^m {}_{t_x^{2i-1}}(t_y t_a^{-1} t_{t_x(y)} t_a^{-1}) \cdot t_{t_x^{2m+1}(y)}t_a^{-1} \cdot t_x^{2m+1} t_b^{-2m-1} t_z^{2m+1} t_c^{-2m-1}. 
\end{align*}
Since $t_dt_at_yt_d$ maps $(y,a)$ to $(a,y)$ by Lemma~\ref{lemma:10001} (1), $t_xt_dt_at_yt_d$ and $t_dt_at_yt_dt_x^{-2m-1}$, denoted $\phi'$ and $\psi'$, map $(y,a)$ and $(t_x^{2m+1}(y),a)$ to $(a,t_x(y))$ and $(a,y)$ by $i(a,x)=0$, respectively. 
Moreover, by Lemma~\ref{lemma:10003}, $t_dt_z \cdot t_et_ct_xt_e \cdot t_bt_d$, denoted $\tau^\prime$, maps $(x,b)$ to $(c,z)$. 
Therefore, for $k=2m,2m+1$, Lemma~\ref{lemma:67} gives
\begin{align*}
t_y t_a^{-1} t_{t_x(y)} t_a^{-1} &= [t_y t_a^{-1}, \phi^\prime], \\
t_{t_x^{2k+1}(y)}t_a^{-1} &= [t_{t_x^{2k+1}(y)}, \psi^\prime], \\
t_x^k t_b^{-k} t_z^k t_c^{-k} &= [t_x^k t_b^{-k}, \tau^\prime]. 
\end{align*}
Since the conjugation of a commutator is a commutator, this finishes the proof. 
\end{proof}
\begin{rem}
The above proof was given in the first draft of \cite{bkm}. Using Proposition~\ref{separatinggenush} it was shown in \cite{bkm} that for a boundary curve $\partial$ of $\Sigma_g^r$, $\cl_{\CM_g^r}(t_\partial^n)=[(n+3)/2]$ if $g\geq 2$ and $r\geq 1$, and therefore, $\scl_{\CM_g^r}(t_\partial)=1/2$. 
\end{rem}

\begin{proof}[Proof of Theorem D (3)]
Suppose that $g\geq 3$ and $h\geq 2$. 
Let $S_h^1$ be the genus-$h$ subsurface of $\Sigma_g$ with one boundary component $s_h$. 
When we write $g=hk+g^\prime$, where $g^\prime=0,1,\ldots,h-1$, there is an element $\rho_k$ in $\CM_g$ such that the subsurfaces $S_h^1, \rho_k(S_h^1), \ldots, \rho_k^{k-1}(S_h^1)$ are disjoint from each other and $\rho_k^k(S_h^1)=S_h^1$. 
In the notation of Proposition~\ref{separatinggenush}, we write 
\begin{align*}
&s_{h,j} := \rho_k^{j-1}(s_h),& && \\
&X_{i,j} := {}_{\rho_k^{j-1}}(X_i),& &Y_{i,j} := {}_{\rho_k^{j-1}}(Y_i)& 
\end{align*}
for $j=1,2,\ldots,k$. Note that $s_{h,k+1}=s_{h,1}$, $X_{i,k+1}=X_{i,1}$ and $Y_{i,k+1}=Y_{i,1}$. 
Then, Proposition~\ref{separatinggenush} and the primitive braid relations give 
\[t_{s_{h,j}}^{n} = [X_{1,j},Y_{1,j}][X_{2,j},Y_{2,j}] \cdots [X_{\left[\frac{|n|+3}{2}\right],j}, Y_{\left[\frac{|n|+3}{2}\right],j}].\]
for $j=1,2,\ldots,k$. 
Since $\rho_k^{j-1}(S_h^1)$ is disjoint from $\rho_k^{j^\prime-1}(S_h^1)$ if $j\neq j^\prime$, any elements $e_j$ in $\CM(\rho_k^{j-1}(S_h^1))$ and any elements $f_{j^\prime}$ in $\CM(\rho_k^{j^\prime-1}(S_h^1))$ satisfy $e_jf_{j^\prime}=f_{j^\prime}e_j$ from the commutative relations. 
Therefore, from Lemma~\ref{lemma:10}, we have 
\[t_{s_{h,1}}^{n} t_{s_{h,2}}^{n} \cdots t_{s_{h,k}}^{n} = [\mathcal{X}_1^\prime,\mathcal{Y}_1^\prime] [\mathcal{X}_2^\prime,\mathcal{Y}_2^\prime] \cdots [\mathcal{X}_{\left[\frac{|n|+3}{2}\right]}^\prime, \mathcal{Y}_{\left[\frac{|n|+3}{2}\right]}^\prime],\]
where $\mathcal{X}_i^\prime = X_{i,1} X_{i,2} \cdots X_{i,k}$ and $\mathcal{Y}_1^\prime = Y_{i,1} Y_{i,2} \cdots Y_{i,k}$. 
Moreover, Lemma~\ref{lemma:8} gives
\[ [ t_{s_{h,1}}^{n} t_{s_{h,2}}^{2n} \cdots t_{s_{h,k-1}}^{(k-1)n}, \rho_k] t_{s_{h,k}}^{kn} = [\mathcal{X}_1^\prime,\mathcal{Y}_1^\prime] [\mathcal{X}_2^\prime,\mathcal{Y}_2^\prime] \cdots [\mathcal{X}_{\left[\frac{|n|+3}{2}\right]}^\prime, \mathcal{Y}_{\left[\frac{|n|+3}{2}\right]}^\prime].\]
In particular, 
\[ t_{s_{h,k}}^{kn} = [ t_{s_{h,1}}^{n} t_{s_{h,2}}^{2n} \cdots t_{s_{h,k-1}}^{(k-1)n}, \rho_k]^{-1} [\mathcal{X}_1^\prime,\mathcal{Y}_1^\prime] [\mathcal{X}_2^\prime,\mathcal{Y}_2^\prime] \cdots [\mathcal{X}_{\left[\frac{|n|+3}{2}\right]}^\prime, \mathcal{Y}_{\left[\frac{|n|+3}{2}\right]}^\prime].\]
Since $X_{i,j},\mathcal{X}'_i$ (resp. $Y_{i,j},\mathcal{Y}'_i$) and $\rho_k$ satisfy the assumption of Lemma~\ref{lemma:21} from their definitions and the primitive braid relations, we obtain $\mathcal{X}_1^\prime \rho_k = \rho_k \mathcal{X}^\prime_1$ (resp. $\mathcal{Y}_1^\prime \rho_k = \rho_k \mathcal{Y}^\prime_1$). 
Note that $s_{h,j}$ is a boundary curve of $\rho_k^{j-1}(S_h^1)$ and that $s_{h,1},s_{h,2},\ldots,s_{h,k}$ are disjoint curves. 
By the property of boundary curves, the commutative relations and Lemma~\ref{lemma:10}, we have
\begin{align*}
[ t_{s_{h,1}}^{n} t_{s_{h,2}}^{2n} \cdots t_{s_{h,k-1}}^{(k-1)n}, \rho_k]^{-1} [\mathcal{X}_1^\prime,\mathcal{Y}_1^\prime] &= [\rho_k, t_{s_{h,1}}^{n} t_{s_{h,2}}^{2n} \cdots t_{s_{h,k-1}}^{(k-1)n}] [\mathcal{X}_1^\prime,\mathcal{Y}_1^\prime] \\
&= [\rho_k \mathcal{X}_1^\prime, t_{s_{h,1}}^{n} t_{s_{h,2}}^{2n} \cdots t_{s_{h,k-1}}^{(k-1)n} \mathcal{Y}_1^\prime], 
\end{align*}
and the proof is complete. 
\end{proof}

\end{document}